\documentclass[11pt,reqno]{amsart}

\usepackage{xcolor}
\usepackage[utf8]{inputenc}
\usepackage[T1]{fontenc}
\usepackage[english]{babel}

\usepackage{microtype}
\usepackage{amsmath,amssymb,amsthm}
\usepackage{amsthm, thmtools, thm-restate}
\usepackage{mathtools} 
\usepackage{bm}        
\usepackage{enumitem}

\usepackage{graphicx}  
\usepackage{caption}
\usepackage{subcaption}
\usepackage{tikz}      
\usepackage{mathrsfs}  
\usepackage{booktabs}  
\usepackage{geometry}
\usepackage{cancel}

\theoremstyle{plain}
\newtheorem{theorem}{Theorem}[section]      
\newtheorem{lemma}[theorem]{Lemma}          
\newtheorem{proposition}[theorem]{Proposition}
\newtheorem{corollary}[theorem]{Corollary}

\newtheorem*{theorem*}{Theorem}             

\theoremstyle{definition}
\newtheorem{definition}[theorem]{Definition}

\newtheorem{remark}[theorem]{Remark}
\usepackage[hypertexnames=false]{hyperref} 


\numberwithin{equation}{section}

\newcommand{\R}{\mathbb{R}}

\usepackage{comment}
\usepackage{graphicx}  
\usepackage{caption}
\usepackage{subcaption}
\usepackage{tikz}      
\usepackage{mathrsfs}  
\usepackage{booktabs}  
\usepackage{geometry}
\usepackage{cancel}

\title[Minkowski mean curvature equation: multiplicity and asymptotics]{Positive and nodal solutions for the Minkowski mean curvature equation: multiplicity and asymptotics}

\author[A. Boscaggin]{Alberto Boscaggin}

\author[F. Colasuonno]{Francesca Colasuonno}

\author[R. Ziegele]{Ricardo Ziegele}
\address{Alberto Boscaggin, Francesca Colasuonno, Ricardo Ziegele \newline \indent 
Dipartimento di Matematica
\newline\indent
Universit\`a di Torino
\newline\indent
via Carlo Alberto 10, 10123 Torino, Italy}
\email{alberto.boscaggin@unito.it}
\email{francesca.colasuonno@unito.it}
\email{ricardoalfonso.ziegelealiaga@unito.it}

\begin{document}

\maketitle

\begin{abstract}
    We consider the Dirichlet problem for the mean curvature operator in Minkowski 
space,
\[
-\operatorname{div}\left(\frac{\nabla u}{\sqrt{1-|\nabla u|^2}}\right) 
= \lambda u + \mu h(x,u) \quad \text{in } \Omega,
\qquad u = 0 \quad \text{on } \partial\Omega,
\]
in a bounded domain $\Omega \subset \mathbb{R}^N$, where $\lambda, \mu$ are real parameters, and the nonlinearity $h$ is superlinear at $u = 0$. In particular, we study the combined effect of the parameters $\lambda,\,\mu$ on the multiplicity of solutions. In the general setting, following Szulkin's approach for nonsmooth functionals, we prove the existence, for $\lambda$ not belonging to the spectrum of the Dirichlet Laplacian and $\mu$ sufficiently large, of a global minimizing solution (with negative action level) and of a min-max solution (with positive action level). Moreover, we characterize the limiting profiles of these solutions as $\mu \to +\infty$. More precisely, when the global minimizer is positive, its limit profile is $\mathrm{dist}(\cdot,\partial\Omega)$, thus saturating, in the limit, the geometric constraint $|\nabla u|\le1$, while min-max solutions collapse uniformly to zero as $\mu\to+\infty$. A nonexistence criterion is also given for suitable values of $\lambda$ and $\mu$. Finally, when the domain $\Omega$ is a ball, using a shooting approach, we establish the existence of arbitrarily many nodal radial solutions, for every $\lambda \ge 0$ and for $\mu$ sufficiently large. 
\end{abstract}

\tableofcontents

\section{Introduction}
In this paper, we are concerned with a boundary value problem of the type 
\begin{equation}\label{eq-prel}
\begin{cases}
\displaystyle -\text{div}\left( \frac{\nabla u}{\sqrt{ 1 - |\nabla u|^2}} \right) = \lambda u + \mu h(x,u)  \qquad &\mbox{in }{\Omega}\\
u=0 &\mbox{on }{\partial\Omega},
\end{cases}
\end{equation}
where $\Omega \subset \mathbb{R}^N$ is a smooth bounded domain, $h: \Omega \times \mathbb{R} \to \mathbb{R}$ is a nonlinear term satisfying $h(\cdot,0) \equiv 0$ and $\lambda, \mu$ are real parameters.

As is well known, the differential operator appearing in the above boundary value problem arises naturally in differential geometry and general relativity: indeed, it can be meant as a prescribed mean curvature operator for $N$-dimensional spacelike Cartesian hypersurfaces in the $(N+1)$-dimensional Lorentz-Minkowski space. It also appears in several models from mathematical physics, such as Born--Infeld type theories in nonlinear electrodynamics. We refer to the seminal works \cite{BaSi82,ChYa76,Ge83} as well as to the more recent contributions \cite{BoCoFo19,BoIa19,ByIkMaMa24, MaMa25} for general results on the subject, especially in the direction of regularity issues. 

From a nonlinear analysis perspective, boundary value problems of type \eqref{eq-prel} have been the subject of extensive investigations during the last few decades. In particular, special efforts have been devoted to the study of the existence of positive solutions to \eqref{eq-prel}, via topological and variational methods. See, among many others, \cite{Az14,Az16,BeJeMa14,BeJeTo13,CoObOmRi13,CoObOmRi13b,Da16,Da17} and the references therein. In the ODE case $N=1$, as well as for the PDE problem on radial domains, the existence of nodal (i.e., sign-changing) solutions has been investigated as well, using shooting arguments \cite{BoGa19,BoCoNo20} and bifurcation theory \cite{DaWa17}.

In order to describe our contributions, and in spite of the fact that much more general nonlinear terms will be allowed, we focus from now on the model problem
\begin{equation}\label{eq-model}
\begin{cases}
\displaystyle -\text{div}\left( \frac{\nabla u}{\sqrt{ 1 - |\nabla u|^2}} \right) = \lambda u + \mu |u|^{p-2}u \qquad &\mbox{in }{\Omega}\\
u=0 &\mbox{on }{\partial\Omega}
\end{cases}
\end{equation}
where $p > 2$. The existence of positive solutions to this problem has been considered both in \cite{CoObOmRi13b} using topological degree theory and in \cite{BeJeMa14} using variational methods. In particular, it was shown in \cite{BeJeMa14} that problem \eqref{eq-model} admits a variational formulation within the framework of Szulkin's nonsmooth critical point theory \cite{Sz86}. In \cite{BeJeMa14}, it was then proved that, for $\lambda \leq 0$ and $\mu$ large enough, problem \eqref{eq-model} has at least two positive solutions: a first one is a global minimizer of the associated action functional (and so, a ground-state solution), while the second one is obtained using the nonsmooth version of the mountain pass theorem. This picture was later greatly improved in the radial case, and when $\lambda = 0$: indeed, it was shown in \cite{BoGa19,DaWa17} that pairs of nodal solutions having more and more zeros appear when $\mu$ becomes larger and larger. 

On the other hand, the problem \eqref{eq-model} for $\mu = 0$ has been considered in \cite{Mi21}. Using an approximation technique together with a $\Gamma$-convergence argument, it is proved therein that the existence of solutions is ensured if and only if $\lambda > \lambda_1$, where $\lambda_1$ is the first eigenvalue of the Dirichlet problem for the Laplacian.

Despite this rich literature, to the best of our knowledge
a systematic investigation of the combined effect of the spectral term $\lambda u $ and of the nonlinear term $\mu \vert u \vert^{p-2} u$ in problem \eqref{eq-model} is still missing. Moreover, since 
the large-$\mu$ regime plays a central role in the available existence results, it seems natural to investigate the asymptotic behaviour of the solutions as $\mu \to +\infty$. The main contributions of the present paper are precisely in these directions and are summarized in the two theorems below.

In the first one, we study existence, multiplicity, and asymptotic behaviour for $\mu \to +\infty$ for the genuine PDE problem.

\begin{theorem}\label{thm:intro}
Let $p>2$. The following holds true:
\begin{itemize}
\item \textbf{Existence and multiplicity.} For every $\lambda \in \mathbb{R}$ not belonging to the spectrum of the Dirichlet Laplacian, there exists $\mu^*(\lambda) > 0$ such that, for every $\mu > \mu^*(\lambda)$, problem \eqref{eq-model} has at least two solutions, $u_\mu^{(s)}$ and $u_{\mu}^{(l)}$, one of which, say $u_\mu^{(l)}$, is positive.
\item \textbf{Asymptotics.} For fixed $\lambda$ and $\mu \to +\infty$, it holds that
$$
u_\mu^{(s)} \to 0, \quad \mbox{ and } \quad u_{\mu}^{(l)} \to \textnormal{dist}(\cdot,\partial\Omega)
$$
uniformly in $\overline\Omega$.
\end{itemize}
\end{theorem}

Our second main result, instead, improves the multiplicity picture when the domain $\Omega$ is the ball $\mathcal{B}_R(0)$ and radial solutions are considered. In the statement, and throughout the paper, we denote by $\lambda_k^{\textnormal{rad}}$, with $k\geq 1$, the $k$-th radial eigenvalue of the Dirichlet Laplacian, and we set by convention $\lambda_0^{\textnormal{rad}} = 0$.

\begin{theorem}\label{thm:radialintro}
Let $\Omega = \mathcal{B}_R(0)$ and $p>2$. There exists a family $\{\mu_j^\star\}_{j= 1}^{+\infty}$ of functions $\mu_j^\star : [0, +\infty) \to [0, +\infty)$, satisfying, for every $j \geq 1,$
\begin{enumerate}
    \item[(i)] $\mu_j^\star(\lambda) = 0$ if $\lambda > \lambda_j^{\textnormal{rad}}$,
    \item[(ii)] $\mu_j^\star(\lambda) \leq \mu_{j+1}^\star(\lambda)$ for all $\lambda \geq0$,
\end{enumerate}
such that problem \eqref{eq-model} has at least $2j-k$ radial solutions when
$$\lambda \in [\lambda_k^{\textnormal{rad}}, \lambda_{k+1}^{\textnormal{rad}}) \quad \text{and} \quad \mu > \mu_j^\star(\lambda) \qquad \text{for some }  j\geq \max\{1,k\}.$$
Precisely, these solutions can be labeled as $\{u_{\mu,\ell}^{(s)}\}_{\ell =k+1}^{j}$, 
$\{u_{\mu,\ell}^{(l)} \}_{\ell = 1}^{j}$ in such a way that
\begin{itemize}
    \item[(a)] $0 < u_{\mu,k+1}^{(s)}(0) < \dots < u_{\mu,j}^{(s)}(0) < u_{\mu,j}^{(l)}(0) < \dots < u_{\mu,1}^{(l)}(0),$
    \item[(b)] $u_{\mu,\ell}^{(s)}$ has exactly $\ell -1$ zeros on $(0,R)$, is strictly decreasing in the interval from $0$ to the first zero, and has exactly one local extremum between any two consecutive zeros, for every $\ell = k+1, \dots, j$,
    \item[(c)] $u_{\mu,\ell}^{(l)}$ has exactly $\ell -1$ zeros on $(0,R)$, is strictly decreasing in the interval from $0$ to the first zero, and has exactly one local extremum between any two consecutive zeros, for every $\ell = 1, \dots, j$.
\end{itemize}
\end{theorem}

Let us now briefly comment on the ideas of the proof. 

The general PDE result, Theorem \ref{thm:intro}, is
obtained through a nonsmooth variational approach in the spirit of Szulkin's
critical point theory, similarly as in the paper \cite{BeJeMa14}. Incidentally, let us mention that the 
nonsmooth character of the problems comes from the fact that the natural energy associated with
\eqref{eq-model} is defined only on the convex set of Lipschitz continuous functions $u$ satisfying the gradient bound $\vert \nabla u\vert \leq 1$, thus preventing the use of classical
critical point arguments. Compared with the result obtained in \cite{BeJeMa14} for $\lambda \leq 0$, the solution denoted by $u_\mu^{(l)}$ turns out to be still a ground-state solution, with negative action level, and can be chosen to be one-signed.
On the other hand, the second solution, $u_{\mu}^{(s)}$, is obtained by a min-max argument. When
$\lambda<\lambda_1$, the geometry is of mountain-pass type, again as in \cite{BeJeMa14}. When, instead, $\lambda > \lambda_1$ and $\lambda$ is not a Dirichlet eigenvalue, the
mountain-pass structure is replaced by a linking geometry associated with the
spectral decomposition of the Dirichlet Laplacian. This is the point where the
spectral parameter $\lambda$ becomes decisive: the position of $\lambda$ with
respect to the Laplacian eigenvalues determines the finite-dimensional
directions along which the linking is constructed. The resulting critical point
has positive energy and is in general expected to be a sign-changing solution of \eqref{eq-model}. In the resonant case $\lambda=\lambda_k$, the conclusion of Theorem \ref{thm:intro} remains valid for $p\leq 4$; see Remark \ref{rem-resonant}. However, owing to the nonlinear nature of the differential operator, establishing a linking geometry in the general case appears to be a challenging problem that warrants further investigation.

The superscript $(s)$ and $(l)$ used for these two solutions stay for ``small'' and ``large'', and are chosen in view of the second part of Theorem \ref{thm:intro} describing their asymptotic behaviour for $\mu \to +\infty$. Indeed, we manage to prove that the min-max solution $u_\mu^{(s)}$  converges uniformly to zero, while the ground-state solution $u_\mu^{(l)}$ converges uniformly to the extremal spacelike profile
$u_\infty(x) := \operatorname{dist}(x,\partial\Omega)$. Incidentally, note that $u_\infty$ is the maximal nonnegative solution (in viscosity sense) of 
the eikonal-type limiting problem
\[
\begin{cases}|\nabla u|=1 \qquad&\text{in } \Omega\\
u=0
&\text{on } \partial\Omega.
\end{cases}\]
Both these convergence results are deduced from the action level characterization of the solutions. Of course, they are valid for every $\lambda \in \mathbb{R},$ thus complementing the results obtained in \cite{BeJeMa14} from the point of view of the asymptotics of the solution when $\mu \to +\infty$.

The proof of Theorem \ref{thm:radialintro} relies, instead, on a shooting technique. Indeed, writing $u=u(r)$, problem \eqref{eq-model} becomes
the one-dimensional singular boundary value problem
\[
\begin{cases}
-\displaystyle\left(
r^{N-1}\dfrac{u'}{\sqrt{1-(u')^2}}
\right)'
=
r^{N-1}\bigl(\lambda u+\mu |u|^{p-2}u\bigr),
& r\in(0,R),\\[2mm]
u'(0)=0,\qquad u(R)=0,
\end{cases}
\]
and a shooting procedure is then performed with respect to the initial value
$u(0)$. Passing to suitable polar coordinates in the phase-plane, the number of
zeros of the corresponding radial solution is encoded by the angular variation
of the shooting trajectory. The argument is inspired by the one in \cite{BoGa19}, dealing with $\lambda = 0$. In this case, the existence of pairs of
radial nodal solutions with an arbitrarily large number of nodal regions comes from the fact that small and large initial data do not rotate enough, whereas intermediate ones perform arbitrarily many rotations as
the parameter $\mu$ grows. In our setting, the
presence of the additional spectral term $\lambda u$ changes the picture. Indeed, for $\lambda=0$, small and large initial data behave in a symmetric way; for $\lambda>0$, instead, small
solutions already feel the linear spectral dynamics, and their rotation number
is governed by the position of $\lambda$ with respect to the radial Dirichlet eigenvalues
$\lambda_k^{\mathrm{rad}}$. This explains the count in
Theorem~\ref{thm:radialintro} : for 
$\lambda\in[\lambda_k^{\mathrm{rad}},\lambda_{k+1}^{\mathrm{rad}})$ and $\mu > \mu^\star_j(\lambda)$, one obtains $j$ large solutions, with nodal levels $1,\ldots,j$, but only
$j-k$ small solutions, corresponding to the nodal levels $k+1,\ldots,j$. Summing up, the two-parameter problem retains the double-gap structure of the one-parameter
case, but the spectral term breaks the full pairwise symmetry.

The paper is organized as follows. In Section \ref{sec2}, we study the general PDE
problem. After recalling the necessary nonsmooth variational setting, we prove a
nonexistence criterion, the existence of a minimizing solution and the existence
of a positive-energy min-max solution, using mountain-pass or linking arguments
according to the position of $\lambda$. We then analyze the asymptotic behaviour
of these two solutions as $\mu\to+\infty$. Section \ref{sec3} is devoted to the radial
problem and to the shooting proof of 
Theorem~\ref{thm:radialintro}, including
the precise nodal count in terms of the radial eigenvalues of the Laplacian.
The appendix presents a technical result used in the shooting arguments.

\section{The PDE problem: a variational approach}\label{sec2}

In this section, we focus on the general PDE problem
\begin{equation}
\label{Plm}
\begin{cases}    
-\mathcal{M}(u) = \lambda u + \mu h(x,u)
    \quad &\text{in } \Omega,\\   u=0  &\text{on } \partial \Omega,
\end{cases}
\end{equation}
where
$$
\mathcal{M}(u):= \text{div}\left( \frac{\nabla u}{\sqrt{ 1 - |\nabla u|^2}} \right),
$$
$\Omega\subset \mathbb R^N$ is a $C^2$ bounded domain (with $N\ge 2$), $\lambda, \mu \in \mathbb{R}$  and $h : \Omega\times\mathbb{R} \to \mathbb{R}$. Below we list the set of assumptions that will be used in our results: 
\begin{itemize}
\item[$(h_1)$] $h$ is a Carath\'eodory function and for every $\varrho>0$ there exists $\alpha_\varrho \in L^\infty(\Omega)$ such that
$$
\ |h(x,s)| \le \alpha_\varrho(x)
\quad
\text{for a.e. } x\in\Omega\text{ and for all } |s|\le \varrho;
$$

\item[$(h_2)$] there exists $\delta >0$ such that $$ h(x,s) s > 0 \quad \text{for a.e. } x\in \Omega \text{ and for all } s \in (-\delta, \delta) \setminus \{0\};$$ 
\item[$(h_3)$] $\displaystyle{\lim_{s\to 0}\frac{H(x,s)} {s^2} }= 0$ uniformly for a.e. $x\in\Omega$, where $H(x,s):=\displaystyle{\int_0^s h(x,\tau)d\tau}$.
\end{itemize}
\begin{remark}\label{rem-h2}
We notice that, since $h(x,\cdot)$ is continuous, $(h_2)$ implies also that $h(x,0)=0$ for a.e. $x\in\Omega$. Moreover, $(h_2)$ also implies that $H(x,s)>0$ if $0<|s|\leq \delta$ and $H(x,s)=0$ if $s=0$, for a.e. $x\in\Omega$. In particular, for every $u\in L^\infty(\Omega)$ with $\|u\|_{L^\infty(\Omega)}\leq \delta$, $H(x,u(x))> 0$ for a.e. $x\in\Omega$ such that $u(x)\neq0$ and $H(x,u(x))= 0$ for a.e. $x\in\Omega$ such that $u(x) = 0$.
\end{remark}

The section is organized as follows. In Section \ref{sec-2.1} we collect some notation, preliminaries, and known results which are useful for the variational formulation of problem \eqref{Plm}. In Section \ref{sec-2.2}, we state and prove our nonexistence, existence and multiplicity results. Finally, in Section \ref{sec-2.3} we analyze the asymptotic behaviour of the solutions found as the parameter $\mu \to +\infty$.

\subsection{Preliminaries}\label{sec-2.1}

Throughout this subsection, only assumption $(h_1)$ is required for $h$. 
Following the approach of \cite{BeJeMa14}, in this paper we restrict our analysis to the so-called \emph{strong strictly spacelike solutions}, cf. \cite{BaSi82} for this terminology. Hence, we refer to a solution of \eqref{Plm} in the sense specified below.  
\begin{definition}
A measurable function $u : \Omega \to \mathbb R$ is a \emph{solution} of \eqref{Plm} if $u\in W^{2,q}(\Omega)$ for some $q>N$, $\|\nabla u\|_{L^\infty(\Omega)} < 1$, it satisfies the equation of \eqref{Plm} a.e. in $\Omega$, and it vanishes on $\partial\Omega$.
\end{definition}
We notice that, since $q>N$, by the Sobolev embeddings, every solution $u$ of \eqref{Plm} is of class $C^1(\overline\Omega)$ and satisfies the boundary condition  $u|_{\partial\Omega}=0$ in the classical sense. 
In the variational formulation of problem \eqref{Plm}, the following closed convex subset of $W^{1,\infty}(\Omega)$ plays a crucial role: 
$$
K_0:=\left\{u\in W^{1,\infty}(\Omega)\,:\, \|\nabla u\|_{L^\infty(\Omega)}\le 1,\, u=0\text{ on $\partial\Omega$}\right\},
$$
where the boundary condition $u=0$ on $\partial\Omega$ holds pointwise, in view of the regularity of $\partial\Omega$ and the Sobolev embedding $W^{1,\infty}(\Omega)\hookrightarrow C^{0,1}(\overline{\Omega})=\mathrm{Lip}(\Omega)$.
Note, in particular, that functions in $K_0$ have Lipschitz constant less than or equal to 1. 
In the next proposition, we recall some known properties of the set $K_0$ that will be useful in the rest of the paper. 
\begin{proposition}[\cite{CoObOmRi13b,BeJeMa14}]\label{prop:propK0}
The following holds true.
\begin{itemize}
    \item[$(i)$]  If $u\in K_0$, then $|u(x)|\le \mathrm{dist}(x,\partial \Omega)$ for every $x\in\Omega$; as a consequence $\|u\|_{L^\infty(\Omega)}\le \frac{1}{2}\mathrm{diam}(\Omega)$. 
    \item[$(ii)$] $K_0$ is compactly embedded in $C(\overline{\Omega})$.
\end{itemize}    
\end{proposition}
\begin{proof}
    (i) If $u\in K_0$, then, from the fact that $u$ is $1$-Lipschitz and that $u|_{\partial \Omega} = 0$, we have 
       $$
|u(x)-u(y)|
\le
|x-y|\quad\text{for every $x,\,y\in\overline{\Omega}$.}
$$
Taking $y\in\partial\Omega$, $u(y)=0$, and so 
$$
|u(x)|\le |x-y|.
$$
Passing to the infimum over $y\in\partial\Omega$
yields $|u(x)|\le \mathrm{dist}(x,\partial \Omega)$ for every $x\in\Omega$. Finally, by simple geometric arguments, $\mathrm{dist}(x,\partial \Omega)\le \frac{1}{2}\mathrm{diam}(\Omega)$ for every $x\in \Omega$.   

Property (ii) is a consequence of the equi-Lipschitz continuity of $K_0$-functions and the Arzelà-Ascoli Theorem, cf. \cite[Lemma 2.2]{BeJeMa14}.
\end{proof}
\begin{remark}
    As a consequence of part (i) of the previous proposition, assumption $(h_1)$ can be weakened in the following form 
\begin{itemize}
\item[$(h'_1)$] $h$ is a Carath\'eodory function for which there exists $\alpha \in L^\infty(\Omega)$ such that
$$
\ |h(x,s)| \le \alpha(x)
\quad
\text{for a.e. } x\in\Omega\text{ and for all } |s|\le \frac{1}{2}\mathrm{diam}(\Omega).
$$
\end{itemize}
\end{remark}
 From Proposition \ref{prop:propK0}, it becomes clear that working in the convex set \(K_0\) offers a twofold advantage: it ensures both compactness and \(L^\infty\) a priori bounds. On the other hand, the operator \(\mathcal{M}\) is not defined for functions \(u \in K_0\) satisfying \(\|\nabla u\|_{L^\infty(\Omega)} = 1\). As already shown in \cite{BeJeMa14}, a Szulkin-type framework proves to be a suitable tool for addressing problem \eqref{Plm} from a variational perspective. Therefore, for $\lambda, \mu \in \R$, we introduce the nonsmooth functional 
$I_{\lambda,\mu} : C(\overline{\Omega}) \to \mathbb R$, defined as
\begin{equation}
\label{funcional}
I_{\lambda,\mu}(u) := \Psi(u) + \Phi_{\lambda,\mu}(u),%\quad\text{for all }u\in C(\overline{\Omega})
\end{equation}
where 
$$\Psi(u) := 
\begin{cases}
        \displaystyle \int_\Omega \left( 1- \sqrt{1-|\nabla u|^2} \right) dx \quad &\text{if } u \in K_0, \smallskip \\
        \displaystyle 
        + \infty &\text{if } u \in C(\overline{\Omega}) \setminus K_0,
\end{cases}
    $$ 
and 
$$\Phi_{\lambda,\mu}(u) := - \frac{\lambda}{2} \int_\Omega u^2 dx - \mu \int_\Omega H(x,u) dx,$$
for all $u\in C(\overline{\Omega})$. 
The functional $I_{\lambda,\mu}$  has the structure required by Szulkin's critical point theory \cite{Sz86}, more precisely: 
\begin{equation*} 
\begin{aligned}  
&I_{\lambda,\mu} = \Psi + \Phi_{\lambda,\mu} \text{ is defined over a real Banach space (namely $(C(\overline{\Omega}),\|\cdot\|_{L^\infty(\Omega)})$)},\\   
&\Psi : C(\overline{\Omega}) \to (-\infty, +\infty] \text{ is convex, proper (i.e., $\Psi \not\equiv +\infty$), and lower semicontinuous},\\
&\Phi_{\lambda,\mu} \in C^1(C(\overline{\Omega});\mathbb{R}), 
\end{aligned} 
\end{equation*} 
see, for instance, {\cite[Lemma 2.4]{BeJeMa14}} for the lower semicontinuity and convexity of $\Psi$.
We now recall the definitions of critical point and Palais-Smale condition for such a nonsmooth functional, in the sense of Szulkin \cite{Sz86}.
\begin{definition}\label{def:crit-pt}
A function $u \in K_0$ is called a \emph{critical point} of $I_{\lambda,\mu}$ if it is a solution of the variational inequality
$$
\Psi(v) - \Psi(u) + \Phi'_{\lambda,\mu}(u)[v-u]  \geq 0 \, \text{ for all } v \in K_0,
$$
namely, for all $v \in K_0$
$$
\int_\Omega \big( 1- \sqrt{1-|\nabla v|^2} \big)dx - \int_\Omega \big( 1- \sqrt{1-|\nabla u|^2} \big)dx - \lambda \int_\Omega u (v-u) dx - \mu \int_\Omega h(x,u)(v-u)dx \geq 0. 
$$
\end{definition}
\begin{definition}\label{def:PS}
$I_{\lambda,\mu}$ satisfies the \emph{Palais-Smale condition} (shortened $(PS)$-condition) if any sequence $\{u_n\} \subset K_0$ with $I_{\lambda,\mu}(u_n)\to c\in\mathbb R$ satisfying 
$$
\Psi(v) - \Psi(u_n) + \Phi'_{\lambda,\mu}(u_n)[v-u_n] \geq -\varepsilon_n\|v-u_n\|_{L^\infty(\Omega)} \, \text{ for all } v \in K_0, 
$$
where $\varepsilon_n \to 0^+$, has a convergent subsequence. 
\end{definition}
As a consequence of Proposition \ref{prop:propK0}-(ii) (cf. also \cite{BeJeMa14}), $I_{\lambda,\mu}$ satisfies the $(PS)$-condition. 

Note that, in principle, a critical point \(u\) of \(I_{\lambda,\mu}\) may satisfy \(\|\nabla u\|_{L^\infty(\Omega)} = 1\), and therefore fail to be a solution of problem \eqref{Plm}. However, combining the regularity estimates in \cite{BaSi82} with a convexity argument, it is shown in \cite{BeJeMa14} that this cannot happen. Precisely, the following result holds. Note that here the sole assumption $(h_1)$ is required, as mentioned before.

\begin{theorem}[{\cite[Theorem 2.1]{BeJeMa14}}]
\label{theo21}
Let $\lambda, \mu \in \R$ and let $h$ satisfy $(h_1)$. A function $u \in C(\overline{\Omega})$ is a critical point of $I_{\lambda,\mu}$ if and only if it is a solution of problem \eqref{Plm}. Moreover, $I_{\lambda,\mu}$ is bounded from below and attains its infimum at some $u_0 \in K_0$, which is a critical point of $I_{\lambda,\mu}$ and, hence, a solution of \eqref{Plm}.
\end{theorem}

Finally, we report here the abstract version of the Linking Theorem for nonsmooth functionals that we will use in the next section to prove the existence of new solutions to \eqref{Plm}. In the following statement, the notions of critical point and $(PS)$-condition are meant analogously as in Definitions \ref{def:crit-pt} and \ref{def:PS}.
\begin{theorem}[{\cite[Theorem 3.4]{Sz86}}, {\cite[Theorem 5.4.29]{PaRa19}}]
\label{linking}
Let $(X,\|\cdot\|)$ be a real Banach space and $I : X \to (-\infty,+\infty]$ be the sum of a convex, proper, lower semicontinuous term and of a term of class $C^1$. 
Let $X = V \oplus W$, where $\dim V < \infty$ and $W$ is closed, and suppose that:
\begin{enumerate}
    \item[$(i)$] $I$ satisfies the $(PS)$-condition;
    \item[$(ii)$] there are constants $\alpha, \rho > 0$ such that 
    $$
        I|_{\partial B_\rho \cap W} \ge \alpha;
    $$
\item[$(iii)$] there is a constant $\mathcal{R}  > 0$ and a function $e \in W$, with $\|e\| > \rho$, such that $I|_{\partial Q} < \alpha$  where $\partial Q$ denotes the boundary of the set
$$Q := (B_{\mathcal{R}} \cap V) \oplus \{ te : 0 \le t \le 1 \}. 
$$
\end{enumerate}
Then $I$ has a critical value $c \ge \alpha$ which may be characterized by
$$
    c = \inf_{f \in \Gamma} \; \sup_{x \in Q} I\bigl(f(x)\bigr),
$$
where $\Gamma := \{\, f \in C(Q;X) : f|_{\partial Q} = \mathrm{id}_{\partial Q} \,\}$.
\end{theorem}

%%%%%%%%%%%%%%%%%%%%%%%%%%%%%%%%%%%%%%%%%%%%%%%%%%%%
\subsection{Nonexistence, existence, and multiplicity of solutions}\label{sec-2.2}

Throughout the section, $\lambda_1< \lambda_2 \le \lambda_3 \le\dots \le \lambda_i \le \dots$ denote the eigenvalues of $-\Delta$ under homogeneous Dirichlet boundary conditions in $\Omega$.

As a first result of this section, we provide a nonexistence criterion. 
\begin{theorem}
   \label{thm:nonexist}
    Let $\lambda\in \mathbb R$, $\mu\ge 0$, and let $h$ satisfy $(h_1)$ and, for some $\xi >0$, 
    \begin{equation} \label{eq:consh3} |h(x,s)| \le |s| \quad\text{for a.e. $x\in\Omega$ and for all $|s|\le \xi$.}\end{equation} 
    Let furthermore $K:=\frac{1}{2}\mathrm{diam}(\Omega)$ and $\alpha_K$ be the $L^\infty$-function given in $(h_1)$ in correspondence of $\varrho = K$. If 
    \begin{equation}
    \label{eq:neccond}        
    \lambda + \mu \left(1+ \frac{\|\alpha_K\|_{L^\infty(\Omega)}}{\xi}\right) \le \lambda_1,
    \end{equation}
then problem \eqref{Plm} admits only the trivial solution $u\equiv 0$. 
\end{theorem}
\begin{proof}
We take inspiration from the last part of the proof of Theorem 1.1 of \cite{Mi21} and recall the following expansion in power series:
\begin{equation}
    \label{eq:series}
1-\sqrt{1-s^2} = \frac{1}{2}s^2 + \sum_{n=2}^\infty a_n s^{2n}\quad\text{for all } s \in [-1,1],
\end{equation}
where $a_n := (2n-3)!!/(2^n n!) > 0$ for every integer $n\ge 2$. Let $u \in K_0$ be a solution of \eqref{Plm}, then $tu \in K_0$ for every $t\in (-1,1)$. Testing the variational inequality in Definition \ref{def:crit-pt} with $v = tu$, $t\in (-1,1)$, and using \eqref{eq:series}, we get for every $t\in (-1,1)$
$$
\frac{1}{2}(t^2-1)\int_\Omega |\nabla u|^2 dx + \sum_{n=2}^\infty a_n (t^{2n}-1)\int_\Omega |\nabla u|^{2n} dx- (t-1) \left[\lambda\int_\Omega u^2 dx+ \mu \int_\Omega h(x,u)u\,dx\right] \ge 0.
$$
Thus, dividing by $t-1(< 0)$ and letting $t\to 1^-$, we obtain
$$
\int_\Omega |\nabla u|^2dx + \sum_{n=2}^\infty a_n  2n \int_\Omega |\nabla u|^{2n} dx\le  \lambda\int_\Omega u^2 dx + \mu \int_\Omega h(x,u)u \,dx.
$$
Hence, by \eqref{eq:consh3}, Proposition \ref{prop:propK0}-(i), assumption $(h_1)$, and the variational characterization of $\lambda_1$ we have
$$
\begin{aligned}
    \int_\Omega |\nabla u|^2dx &+ \sum_{n=2}^\infty a_n  2n \int_\Omega |\nabla u|^{2n}dx \\
    &\le  \lambda \int_\Omega u^2dx + \mu \left(\int_{\Omega\cap \{|u|\le \xi\}} h(x,u)u\,dx + \int_{\Omega\cap \{|u|> \xi\}} h(x,u)u\,dx\right)\\
    &\le \lambda \int_\Omega u^2\, dx + \mu \left(\int_\Omega u^2\,dx + \int_{\Omega\cap \{|u|> \xi\}} \frac{\alpha_K(x)}{\xi}u^2\,dx\right)\\
    &\le \lambda \int_\Omega u^2\,dx + \mu \left(1 + \frac{\|\alpha_K\|_{L^\infty(\Omega)}}{\xi}\right)\int_\Omega u^2\,dx\\
    &\le \left[\lambda + \mu \left(1 + \frac{\|\alpha_K\|_{L^\infty(\Omega)}}{\xi}\right)\right]\frac{1}{\lambda_1}\int_\Omega |\nabla u|^2\,dx,
\end{aligned}
$$
where we have denoted $\{|u|\le \xi\}:=\{x\in\Omega\,:\,|u(x)|\le \xi\}$ and similarly for $\{|u|> \xi\}$.
Therefore, since $a_2=1/8$ and using that $a_n > 0$ for all $n\ge 2$, by \eqref{eq:neccond} we get 
$$
\begin{aligned}
\frac{1}{2}\int_\Omega |\nabla u|^4\,dx &\le \sum_{n=2}^\infty a_n  2n \int_\Omega |\nabla u|^{2n} \,dx\le  \left\{\left[\lambda + \mu \left(1 + \frac{\|\alpha_K\|_{L^\infty(\Omega)}}{\xi}\right)\right]\frac{1}{\lambda_1}-1\right\}\int_\Omega |\nabla u|^2 \,dx\le 0.
\end{aligned}
$$
Since $K_0 \hookrightarrow W^{1,4}_0(\Omega)$, this implies that $u\equiv 0$ and concludes the proof.
\end{proof}

\begin{remark} For the prototype nonlinearity $h(x,s)=|s|^{p-2}s$, with $p>2$, we have $\xi=1$ in \eqref{eq:consh3} and $\alpha_K= K^{p-1}$ in $(h_1)$, so that \eqref{eq:neccond} specifies into 
$$\lambda + \mu \left(1+ \left(\frac{1}{2}\mathrm{diam}(\Omega)\right)^{p-1}\right) \le \lambda_1.
$$
\end{remark}

In contrast with Theorem \ref{thm:nonexist}, the next result shows that, when $\mu$ is large enough, \eqref{Plm} always admits a nontrivial solution provided by global minimization (i.e., a ground-state solution). Incidentally, it is worth observing that for this theorem, assumption $(h_3)$ is not needed.

\begin{theorem}
\label{thm:ground}
Let $\lambda \in \mathbb{R}$ and let $h$ satisfy $(h_1)$ and $(h_2)$. Then, there exists $\widetilde{\mu}=\widetilde{\mu}(\lambda) \in \mathbb{R}$ such that for every $\mu>\widetilde{\mu}$ problem \eqref{Plm} admits a nontrivial solution $u^{(l)}_{\lambda,\mu}$ such that
\begin{equation}
    \label{eq:infI}
I_{\lambda,\mu}(u^{(l)}_{\lambda,\mu}) = \inf_{u\in C(\overline{\Omega})} I_{\lambda,\mu}(u) < 0.
\end{equation}
If $h$ is odd with respect to $u$, then $u^{(l)}_{\lambda,\mu}$ can be chosen to be one-signed (and thus both $u^{(l)}_{\lambda,\mu}$ and $-u^{(l)}_{\lambda,\mu}$ are solutions). 
\end{theorem}

\begin{proof} By Theorem \ref{theo21}, we know that the infimum in \eqref{eq:infI} is achieved at some $u^{(l)}_{\lambda,\mu}\in K_0$ which is a critical point of $I_{\lambda,\mu}$ and so a solution of \eqref{Plm}. It remains to prove that $u^{(l)}_{\lambda,\mu}\not\equiv 0$ when $\mu > \widetilde\mu(\lambda)$ exhibited below.  

First, arguing as in the proof of \cite[Theorem 3.1]{BeJeMa14}, let $x_0 \in \Omega$, $\delta>0$ be as in $(h_2)$, $0<r_0<\delta$ be such that $\overline{B}_{r_0}(x_0) \subset \Omega$, and let 
    $$\eta(x) :=
\begin{cases} \exp\left( \frac{r_0^2}{|x-x_0|^2 - r_{0}^2} \right) &\text{if } x \in B_{r_0}(x_0),\\
0 & \text{if } x \in \Omega \setminus B_{r_0}(x_0).
\end{cases}$$
With this, we define $$ \eta_0 (x) := \min\{ r_0, \|\nabla \eta \|^{-1}_{L^\infty(\Omega)}\}\eta(x).$$
One has, by construction, that $\eta_0 \in K_0$ and $0 \leq \eta_0 < \delta$, which together with Remark \ref{rem-h2} yields
$$\int_\Omega H(x,\eta_0)\, dx >0.$$
Then, choosing
\begin{equation*}
\label{mu1*}
\widetilde\mu := \frac{\Psi(\eta_0) - \frac{\lambda}{2} \int_{\Omega} \eta_0^2\,dx}{\int_{\Omega} H(x, \eta_0)\,dx},
\end{equation*}
and noticing that, for all $\mu > \widetilde \mu$ one has $I_{\lambda,\mu}(\eta_0) < I_{\lambda,\widetilde\mu}(\eta_0) = 0$, the conclusion is immediate since $I_{\lambda,\mu}$ is bounded from below, cf. Theorem \ref{theo21}, and $I_{\lambda,\mu}(0)=0$. 

Finally, as for the last part of the statement, it is enough to observe that, if $h(x,\cdot)$ is odd, then $I_{\lambda,\mu}(u)=I_{\lambda,\mu}(|u|)=I_{\lambda,\mu}(-|u|)$ for all $u\in C(\overline{\Omega})$, this allows one to replace $u$ with $\pm|u|$ in the variational problem in \eqref{eq:infI} to get a one-signed solution. 
\end{proof}

\begin{remark}
When $h$ is not odd in $u$, one can apply the previous result to both $\tilde{h}(x,u) = \textnormal{sgn}(u)h(x,|u|)$ and $\tilde{h}(x,u) = - \textnormal{sgn}(u)h(x,-|u|)$, so as to obtain a positive and a negative solution with a negative energy level.
\end{remark}

\begin{remark}
We notice that, at least when $(h_2)$ holds with $\delta \ge \tfrac{1}{2}\mathrm{diam}(\Omega)$, we can easily prove that, if $\lambda > \lambda_1$,  problem \eqref{Plm} admits a nontrivial solution $u^{(l)}_{\lambda,\mu}$  for every $\mu\ge 0$, that is, we can take $\mu^\star(\lambda) < 0$. Indeed, in this case, it is sufficient to recall that in \cite[Theorem 1.1]{Mi21} it is proved that, for any $\lambda > \lambda_1$, the eigenvalue problem
$$\begin{cases}
-\mathcal{M}(u) = \lambda u &\text{ in }\Omega,\\
u = 0 &\text{ on }\partial\Omega
\end{cases}
$$
has a solution $u_{\lambda,0} \geq 0$ such that $I_{\lambda,0}(u_{\lambda,0}) < 0$. Since $\mu \geq 0$ and $H(x,u_{\lambda,0}) \ge 0$ by Remark \ref{rem-h2} and Proposition \ref{prop:propK0}-(i), it follows that $I_{\lambda,\mu}(u_{\lambda,0}) = I_{\lambda,0}(u_{\lambda,0}) - \mu \int_\Omega H(x,u_{\lambda,0})dx <0$. Then we can conclude as in the proof of the previous theorem.
\end{remark}

 We now present our main result of this section, which shows that, when $\lambda \in \R$ does not belong to the spectrum of the Dirichlet Laplacian (when $\lambda$ is an eigenvalue, a further assumption on $h$ is required, see Remark \ref{rem-resonant}), and $\mu$ is large enough, a solution with positive energy level can be provided via a min-max argument.

\begin{theorem}
\label{theoexistence}
Let $\lambda\in\mathbb{R}$ lie outside the spectrum of the Dirichlet Laplacian, and let $h$ satisfy $(h_1)$, $(h_2)$, and $(h_3)$. Then, there exists  $\bar{\mu}=\bar{\mu}(\lambda)\ge0$ such that for every $\mu > \bar{\mu}$ problem \eqref{Plm} admits a non-trivial solution $u^{(s)}_{\lambda,\mu}$ such that
$$
I_{\lambda,\mu}(u^{(s)}_{\lambda,\mu}) > 0.
$$
\end{theorem}

\begin{proof} 
Set by convention $\lambda_0=-\infty$ and $\varphi_0\equiv 0$ and let $\varphi_1, \varphi_2, \varphi_3,\dots,\varphi_i,\dots$ be an $L^2$-orthonormal basis of eigenfunctions corresponding to the eigenvalues $0<\lambda_1 < \lambda_2 \le \lambda_3 \le  \dots \le \lambda_i\le \dots$ of $-\Delta$ under homogeneous Dirichlet boundary conditions in $\Omega$. 
Note that 
\begin{equation}
    \label{eq:ONbasis-prop}
\varphi_i\in H^1_0(\Omega),\quad\int_\Omega \varphi_i^2 dx=1,\quad \text{and}\quad\int_\Omega |\nabla \varphi_i|^2 dx=\lambda_i\quad\text{for every $i\in\mathbb N$.}
\end{equation}
Let $k=0$ if $\lambda<\lambda_1$ and let $k\ge 1$ be the integer such that $\lambda\in (\lambda_k,\lambda_{k+1})\neq\emptyset$ if $\lambda >\lambda_1$.
Following the notation in Theorem~\ref{linking}, we set $X:=C(\overline \Omega)$,
\[
V:=\operatorname{span}\{\varphi_1,\dots,\varphi_k\},
\quad \text{and}\quad
W := \left\{u\in X : \int_\Omega u\varphi_i\,dx=0, \, i= 1,\dots, k\right\}.
\]
In view of \eqref{eq:ONbasis-prop}, for every $v\in V$ we have that $v\in H^1_0(\Omega)$ and 
\begin{equation}
    \label{eq:vinV}
v=\sum_{i=1}^{k}a_i \varphi_i, \quad \int_\Omega v^2 dx = \sum_{i=1}^k a_i^2,\quad \int_\Omega |\nabla v|^2 dx = \sum_{i=1}^k a_i^2\lambda_i,
\end{equation}
for suitable real numbers $a_i$. 
Similarly, for every $w\in W\cap K_0$,
\[w=\sum_{i=k+1}^{\infty}a_i \varphi_i, \quad \int_\Omega w^2 dx = \sum_{i=k+1}^\infty a_i^2,\quad \int_\Omega |\nabla w|^2 dx = \sum_{i=k+1}^\infty a_i^2\lambda_i,\]
where $a_i$ are suitable real numbers. Moreover, as a consequence of the $L^2$-orthogonality of the basis $\{\varphi_i\}$, for every $v\in V$ and $w\in W \cap K_0$ it holds
\begin{equation}
    \label{eq:orthog}
    \int_\Omega v\,w \,dx=0 \quad\text{and}\quad \int_\Omega \nabla v\,\nabla w \,dx=0.
\end{equation}
Furthermore, since $V$ is finite dimensional, all norms are equivalent therein, and so \[\|\nabla v\|_{L^\infty(\Omega)}\le C \|v\|_{L^\infty(\Omega)}\] for some positive $C$ independent of $v\in V$.  
Note that, if $k=0$, $V=\{0\}$ and $W=X$ (namely, the linking geometry reduces to the mountain pass one).

Now, we verify that the assumptions in Theorem \ref{linking} are satisfied for the functional $I_{\lambda,\mu}$. As already mentioned, assumption (i) of Theorem \ref{linking}, that is, the Palais-Smale condition, is satisfied as a consequence of Proposition \ref{prop:propK0}-(ii). 

As for condition (iii), we will prove that there exist $\mathcal{R}>0$ and $e \in W \cap K_0\setminus\{0\}$, such that
\begin{equation} \label{condition2}
I_{\lambda,\mu}(u) \le 0 \, \text{ for all } u \in \partial Q, \, \, Q:= \{ u = v+te \,:\, v \in V, \,\|v\|_{L^\infty(\Omega)} < \mathcal{R}, \, t\in [0,1] \}.\end{equation} 
Notice that, if $k=0$, $\partial Q = \{0\}\cup\{e\}$ and $I_{\lambda,\mu}(0)=0<\alpha$. Taking $e$ such that $\|e\|_{L^\infty(\Omega)}\in (0,\delta/2)$, we have by Remark \ref{rem-h2} that $\int_{\Omega} H(x,e)\, dx>0$ and so 
\begin{equation*}
I_{\lambda,\mu}(e) = \int_{\Omega} (1-\sqrt{1-|\nabla e|^2})\, dx - \frac{\lambda}{2} \int_{\Omega} e^2\, dx  - \mu \int_{\Omega} H(x,e)\, dx\le 0
\end{equation*}
for every $\mu \ge \bar\mu$, with  
        \[
        \bar\mu = \bar\mu(\lambda) := \max\left\{0,\frac{\displaystyle{\int_{\Omega} (1-\sqrt{1-|\nabla e|^2})\, dx - \frac{\lambda}{2} \int_{\Omega} e^2\, dx}}{\displaystyle{\int_{\Omega} H(x,e)\, dx}}\right\}.
        \] 
If $k\ge1$, let 
\[
\begin{gathered}
\varepsilon=\varepsilon(\lambda):=\frac{\lambda}{\lambda_k}-1,\quad s_{\varepsilon}=s_{\varepsilon}(\lambda)=:=\frac{2\sqrt{\varepsilon}}{1+\varepsilon},\quad \mathcal R=\mathcal R(\lambda):=\min\left\{\frac{\delta}{2},\frac{1}{2C},\frac{s_\varepsilon}{C}\right\},
\\
e\in W\cap K_0\text{ such that $\|e\|_{L^\infty(\Omega)}\in\left(0, \frac\delta2\right)$ and $\|\nabla e\|_{L^\infty(\Omega)}\le \frac12$,}
\end{gathered}
\]
and
\[
\begin{aligned}\partial Q = & \ell_1\cup\ell_2\cup\ell_3\\&:=\{ v \, \in V\, : \, \|v\|_{L^{\infty}(\Omega)} < \mathcal{R} \}\cup \{v +  e \, : \,v \in V,\, \|v\|_{L^{\infty}(\Omega)} < \mathcal{R}\}\\ &\hspace{4cm}\cup \{ v+te \, : v \in V,\, \|v\|_{L^{\infty}(\Omega)} = \mathcal{R}, \, t\in [0,1]\}.
\end{aligned}
\]
Hence, for all $v$ such that $u=v+te\in \partial Q$, using that $\|v\|_{L^\infty(\Omega)} \le  \mathcal R < \delta$, we have 
\[
H(x,v(x))\ge 0\quad\text{for a.e. } x\in \Omega,
\] 
in view of Remark \ref{rem-h2}. Moreover, 
\[
\|\nabla v\|_{L^\infty(\Omega)}\le  C\mathcal R \le \min\left\{\frac12,s_\varepsilon\right\},
\] 
in particular, $v\in K_0$.
In order to prove \eqref{condition2}, let us begin with $\ell_1$. Since $v\in K_0$ and 
\[1-\sqrt{1-s^2}\le \frac{1+\varepsilon}{2}s^2\quad\text{for all }|s|\le s_\varepsilon,
\]
using \eqref{eq:vinV},  we have, for every $v\in\ell_1$,
$$
\begin{aligned}
I_{\lambda,\mu}(v) &\le \int_\Omega \big(1-\sqrt{1-|\nabla v|^2} \big) \, dx- \frac \lambda 2 \int_\Omega v^2 \,dx \\
&\leq \frac{1+\varepsilon}{2}\int_\Omega |\nabla v|^2 \, dx - \frac \lambda 2 \int_\Omega v^2 \,dx = \frac{1+\varepsilon}{2}\sum_{i=1}^k a_i^2\lambda_i  - \frac \lambda 2 \int_\Omega v^2 \,dx\\
&\le \frac{1}{2}((1+\varepsilon)\lambda_k-\lambda)\int_\Omega v^2 dx = 0 \qquad \text{for all } \mu > 0.
\end{aligned}
$$
We now consider the sides 
        \[\begin{gathered}\ell_2 = \{u = v +  e \,:\, v\in V,\, \|v\|_{L^{\infty}(\Omega)} < \mathcal{R}\}\quad\text{and}\\
        \ell_3=\{ u = v+te \,:\,  v\in V,\, \|v\|_{L^{\infty}(\Omega)} = \mathcal{R}, \, t\in [0,1]\}.
        \end{gathered}
        \] 
For every $u$ on these sides, 
        \[0<\|u\|_{L^\infty(\Omega)}\le\|v\|_{L^\infty(\Omega)}+t\|e\|_{L^\infty(\Omega)}\le \mathcal R + \|e\|_{L^\infty(\Omega)},\]
        where the first strict inequality follows from the fact that $V\cap W=\{0\}$. 
         Thus, using that $\mathcal R \le \delta/2$, we get from the estimate above,
        \begin{equation}
            \label{eq:0<u<delta}   0<\|u\|_{L^\infty(\Omega)} < \delta\quad\text{for every $u\in\ell_2\cup\ell_3$.}
        \end{equation}
        It follows from Remark \ref{rem-h2} that $\int_\Omega H(x,u) dx >0$ for every $u \in \ell_2 \cup \ell_3$. Moreover, by the choice of $e$, we also infer that, for any \(u\in\ell_2\cup\ell_3\) 
        \[\|\nabla u\|_{L^\infty(\Omega)}\le \|\nabla v\|_{L^\infty(\Omega)} + \|\nabla e\|_{L^\infty(\Omega)} \le 1,\] 
        that is $u\in K_0$ and consequently
        \[
        \Psi(u)= \int_\Omega\left(1-\sqrt{1-|\nabla u|^2}\right)dx\le |\Omega|\quad\text{for all }u\in \ell_2\cup\ell_3.\] This yields, 
        $$\begin{aligned}
            I_{\lambda,\mu}(u) &\leq |\Omega| - \frac \lambda 2 \int_\Omega u^2 \, dx - \mu \int_{\Omega} H(x,u) \, dx \\
            &< |\Omega| - \mu \int_{\Omega} H(x,u) \,dx\quad\text{for all }u\in \ell_2\cup\ell_3.
        \end{aligned}$$
        We now claim that 
\begin{equation}
    \label{eq:claim-infH}
\inf_{v+t e\,\in\, \ell_2\cup\ell_3}\int_{\Omega} H(x,v+te)\, dx > 0.
\end{equation}
Indeed, 
\[
\inf_{v+t e\,\in\, \ell_2\cup\ell_3}\int_{\Omega} H(x,v+te)\, dx = \min\left\{\inf_{v+t e\,\in\, \ell_2}\int_{\Omega} H(x,v+e)\, dx,\inf_{v+t e\,\in\, \ell_3}\int_{\Omega} H(x,v+te)\, dx\right\}.
\]
We can separately estimate the two infima over $\ell_2$ and over $\ell_3$. As for $\ell_2$, using the continuity of the functional $v\in V\mapsto \int_\Omega H(x,v) dx$ and that $V$ is finite dimensional, by the compactness of the closed ball $\{v\in V\,:\, \|v\|_{L^\infty(\Omega)}\le \mathcal R\}$ in $V$, we get 
\[
\begin{aligned}
\inf_{v+e\,\in\, \ell_2}\int_{\Omega} H(x,v+e) \, dx&
%= \inf_{v\in V,\,\|v\|_{L^\infty(\Omega)}<\mathcal R}\int_{\Omega} H(x,v+e)\\ 
\ge \inf_{v\in V,\,\|v\|_{L^\infty(\Omega)}\le\mathcal R}\int_{\Omega} H(x,v+e)\, dx \\
& = \min_{v\in V,\,\|v\|_{L^\infty(\Omega)}\le\mathcal R}\int_{\Omega} H(x,v+e)\, dx =: m_{\ell_2} >0,
\end{aligned}
\]
where the last inequality follows from the fact that $v+e\not\equiv 0$ and Remark \ref{rem-h2}.
Similar considerations over $\ell_3$ lead to 
\[
\begin{aligned}
\inf_{v+t e\,\in\, \ell_3}\int_{\Omega} H(x,v+te) \, dx&= \min_{v\in V,\,\|v\|_{L^\infty(\Omega)}=\mathcal R,\,t\in[0,1]}\int_{\Omega} H(x,v+te)\, dx =: m_{\ell_3} >0,
\end{aligned}
\]
as now $v\not\equiv 0$ in $V$ and $te\in W$. Therefore, the claim \eqref{eq:claim-infH} is proved. As a consequence, it is possible to define 
\begin{equation}
    \label{def:barmu}
\bar\mu=\bar\mu(\lambda):=\frac{|\Omega|}{\min\{m_{\ell_2},m_{\ell_3}\}},
\end{equation}
to get $I_{\lambda,\mu}(u)\le 0$ for every $u\in \partial Q$ and for every $\mu \ge \bar \mu$. 

This proves that condition (iii) of Theorem \ref{linking} is satisfied by $I_{\lambda,\mu}$ for $\mu$ sufficiently large. Notice that the dependence of $\bar\mu$  on $\lambda$ is hidden in the choice of $\mathcal R$, which determines $\ell_2$ and $\ell_3$.

As for condition (ii) of Theorem \ref{linking}, let $w \in W\cap K_0$ and treat all the cases $ k\ge 0$ simultaneously. By $(h_3)$ and in view of the embedding $L^\infty(\Omega)\hookrightarrow L^2(\Omega)$, for every $\mu>0$ there exists $r_\mu=r_\mu(\lambda)>0$ such that 
\[H(x,w(x))\le \frac{\lambda_{k+1}-\lambda}{4\mu}w(x)^2\quad \text{for a.e. $x\in\Omega$ and all $\|w\|_{L^\infty(\Omega)}\le r_\mu$}.\]
Using also 
\[
1-\sqrt{1-s^2} \geq \frac{1}{2} s^2,
\]
it follows that
        \begin{equation}
        \label{eq:(ii)}    
        \begin{aligned}
        I_{\lambda,\mu}(w) &\geq \frac{1}{2} \int_{\Omega} |\nabla w|^2\,dx - \frac{\lambda}{2} \int_{\Omega} w^2\,dx - \mu \int_\Omega H(x,w)\,dx\\ 
        &= \frac{1}{2} \sum_{i=k+1}^\infty (\lambda_i - \lambda)a_i^2 - \mu \int_\Omega H(x,w)\,dx\\
        &\geq \frac{\lambda_{k+1} - \lambda}{2} \|w\|_{L^2(\Omega)}^2 - \mu \frac{\lambda_{k+1} - \lambda}{4\mu}\|w\|_{L^2(\Omega)}^2=\frac{\lambda_{k+1} - \lambda}{4} \|w\|_{L^2(\Omega)}^2\quad\text{for all $\mu>0$.}
        \end{aligned}
        \end{equation}
        Let 
        \[\rho=\rho(\lambda,\mu):=\min\left\{r_\mu, \frac{\delta}{2},\|e\|_{L^\infty(\Omega)}\right\}.\]
        Since $\lambda < \lambda_{k+1}$ and $I_{\lambda,\mu}\equiv +\infty$ in $X\setminus K_0$, it remains to prove that 
        \begin{equation} 
        \label{infiii}
        \mathfrak{i}_{\rho}:=\inf_{\{w \in W\cap K_0\,:\, \|w\|_{L^{\infty}(\Omega)} = \rho\}} \int_{\Omega} w^2\,dx >0.
        \end{equation}
        Arguing by contradiction, suppose that there exists a sequence $\{w_n\} \subset W\cap K_0$  such that 
        $$ \|w_n\|_{L^{\infty}(\Omega)} = \rho,\, \quad\int_\Omega w_n^2\,dx \to 0.$$
      The sequence $\{w_n\}$ is bounded in $K_0$, therefore, by Proposition \ref{prop:propK0}, there exists a subsequence converging uniformly to some function $w \in C(\overline{\Omega})$. By the continuity of the $L^\infty$-norm, we have $\|w\|_{L^\infty(\Omega)} = \rho$ and, passing to the limit under the integral sign by uniform convergence, we have 
        $$\int_\Omega w_n^2\,dx \to \int_{\Omega} w^2\,dx.$$
        Hence $\int_\Omega w^2 \, dx = 0$, which implies $w \equiv 0$ in $\Omega$, contradicting $\|w\|_{L^\infty(\Omega)} = \rho>0$. This contradiction proves \eqref{infiii}. As a consequence, by \eqref{eq:(ii)} there exists 
    \[
\alpha=\alpha(\lambda,\mu):=\frac{\lambda_{k+1}-\lambda}{4}\mathfrak{i}_\rho>0
        \]
such that
$$
I_{\lambda,\mu}\big|_{\partial B_\rho \cap W} \ge \alpha.
$$
 In conclusion, by Theorem \ref{linking}, we get the existence of a critical point $u_{\lambda,\mu}$ of $I_{\lambda,\mu}$ at the critical value 
\[
c_{\lambda,\mu}:=\inf_{f\in \Gamma}\max_{u\in Q}I_{\lambda,\mu}(f(u))\ge \alpha_{\lambda} > 0,\quad\text{where }\Gamma:=\left\{f\in C(Q;X)\,:\,f|_{\partial Q}=\mathrm{id}_{\partial Q}\right\},
\] 
for every $\lambda\neq \lambda_k$, provided that $\mu>\bar\mu$.
\end{proof}

\begin{remark}\label{rem-resonant} If $\lambda = \lambda_k$ for some integer $k \geq 1$, the conclusion of Theorem \ref{theoexistence} is still true if, for some $M > 0$, we have
\begin{equation}
\label{ass-resonant}
 H(x,s) \ge M |s|^4, \quad \text{for a.e. } x\in \Omega \text{ and $\vert s \vert \leq \delta$. }
\end{equation}
Indeed, a careful checking of the proof shows that, keeping the notation used therein, the only point where the assumption $\lambda > \lambda_k$ is actually needed is in the verification that
$I_{\lambda,\mu}(u) \leq 0$ for every $u \in \ell_1$. 
If, however, condition \eqref{ass-resonant} is satisfied, this can be shown to be true by arguing as follows: first, exploiting the Taylor expansion of $s \mapsto 1-\sqrt{1-s^2}$ for $s\in [-1,1]$ we have that, for every $v \in V$ such that $\|v\|_{L^\infty (\Omega)} \leq \mathcal R$, with $\mathcal R:= \min \left\{\delta/2, 1/2C \right\} $,
$$\Psi(v) = \frac 12 \int_\Omega |\nabla v|^2 \,dx+ \frac{1}{8} \int_\Omega |\nabla v|^4 \, dx+  o\left(\|v\|_{L^\infty(\Omega)}^4\right).$$
Then, for every $v \in \ell_1$, writing $\int_\Omega v^2\,dx=\sum_{i=1}^k a_i^2$, we have
$$ 
\begin{aligned}
I_{\lambda_k,\mu}(v) &=\frac 12 \int_\Omega |\nabla v|^2 \, dx+ \frac{1}{8} \int_\Omega |\nabla v|^4 \, dx  - \frac{\lambda_k}2 \int_\Omega v^2 dx - \mu \int_\Omega H(x,v) \, dx  +o\left(\|v\|_{L^\infty(\Omega)}^4\right)\\
&\le \frac{1}{2}\sum_{i=1}^k  a_i^2 \lambda_i - \frac{\lambda_k}2 \int_\Omega v^2 dx + \frac{1}{8} \int_\Omega |\nabla v|^4 \, dx  - \mu M\int_\Omega v^4 \, dx +  o\left(\|v\|_{L^\infty(\Omega)}^4\right)\\
& \le \frac{1}{8} \int_\Omega |\nabla v|^4 \, dx  - \mu MC'\int_\Omega |\nabla v|^4 \, dx +  o\left(\|\nabla v\|_{L^4(\Omega)}^4\right)\le 0
\end{aligned}
$$
for all $\mu>\bar \mu:=1/8MC'$, where $C'$ is the constant arising from the equivalence of the norms $\|\cdot\|_{L^4(\Omega)}$ and $\|\nabla \cdot\|_{L^4(\Omega)}$ in the finite dimensional space $V$.

Notice that, for the prototype nonlinearity $h(x,s)=|s|^{p-2}s$, this result applies for all $p\in (2,4]$.  
\end{remark}

By combining Theorem \ref{thm:ground} with Theorem \ref{theoexistence}, we get the following multiplicity result. In particular, the choice $h(x,s)=|s|^{p-2}s$ (with $p > 2$) provides the multiplicity result stated in the Introduction. 

\begin{corollary}\label{cor:en-levels}
Let $\lambda \in \mathbb{R}$, and let $h$ satisfy $(h_1)$, $(h_2)$ and $(h_3)$. If $\lambda$ lies outside the spectrum of the Dirichlet Laplacian or if \eqref{ass-resonant} holds, then there exists $\mu^\star=\mu^\star(\lambda):=\max\{\widetilde{\mu},\bar\mu\} \in (0,\infty)$ such that for every $\mu>\mu^\star$ problem \eqref{Plm} admits two nontrivial solutions $u^{(s)}_{\lambda,\mu}$ and $u^{(l)}_{\lambda,\mu}$ such that
$$
I_{\lambda,\mu}(u^{(l)}_{\lambda,\mu}) < 0 < I_{\lambda,\mu}(u^{(s)}_{\lambda,\mu}) .
$$
If $h$ is odd with respect to $u$, the solution $u^{(l)}_{\lambda,\mu}$ can be chosen to be one-signed.
\end{corollary}

\subsection{Asymptotic behaviour as $\mu \to +\infty$}\label{sec-2.3}
In this section, as we are interested in the limit of solutions as $\mu\to+\infty$, we will drop the subscript $\lambda$ and label the solutions, as well as the energy functional, only in terms of $\mu$.
For the next results, we introduce a stronger form of $(h_2)$, namely:
\begin{itemize}
    \item[$(h_2')$] $h(x,s)s > 0$ for a.e. $x\in \Omega$ and for all $s \neq 0$.
\end{itemize}
In view of Proposition \ref{prop:propK0}-(i), it is alternatively possible to require $(h_2)$ and $\tfrac{1}{2}\mathrm{diam}(\Omega)\le \delta$.

In the next proposition, we prove in particular that the min-max solutions tend uniformly to zero.
\begin{proposition}
    Assume that $(h_1)$ and $(h_2')$ hold. Let $\lambda \in \mathbb R$ be fixed, and $M>0$ be such that, for every $\mu\ge M$, $u_{\mu}$ is a solution of \eqref{Plm} and  
    \begin{equation}
        \label{eq:liminf}
    \liminf_{\mu\to +\infty}I_{\mu}(u_{\mu})>-\infty.
    \end{equation}
       Then, 
$$
u_{\mu}
\to 0
\quad \text{uniformly in } \overline{\Omega}
\quad \text{as } \mu\to+\infty.
$$
In particular, under the further assumption $(h_3)$ on the nonlinearity $h$ and when $\lambda$ is not an eigenvalue of the Dirichlet Laplacian, the above conclusion holds for any family of min-max solutions $\{u_{\mu}^{(s)}\}_{\mu\ge \bar\mu}$ as in Theorem \ref{theoexistence}. 
\end{proposition}

\begin{proof}
    By the compactness of the cone $K_0$ in $C(\overline{\Omega})$ (cf., Proposition \ref{prop:propK0}-(ii)), we know that there exists a function $u_\infty\in K_0$ such that, up to a subsequence, $u_{\mu} \to u_\infty$ in $C(\overline{\Omega})$ as $\mu \to +\infty$. Arguing by contradiction, let us assume that $u_\infty \not \equiv0$. Then, by $(h_1)$ and $(h'_2)$, 
$$
\lim_{\mu\to +\infty}\int_\Omega H(x,u_{\mu})\, dx = \int_\Omega H(x,u_{\infty})\, dx > 0.
$$
Therefore, since $\Psi(u_\mu)\le |\Omega|$ and $\{u_\mu\}_{\mu\geq M}$ is bounded in $L^2(\Omega)$ for every $\mu$ large enough, we get for some $C>0$ and for all $\mu$ large
    $$I_{\mu}(u_\mu)=\Psi(u_{\mu}) - \frac\lambda2 \int_{\Omega} u_{\mu}^2\, dx - \mu \int_{\Omega}H(x,u_{\mu})\, dx \leq C - \mu \int_{\Omega}H(x,u_{\mu})dx.$$
Hence, passing to the limit in $\mu$, we get
\[
\lim_{\mu\to +\infty}I_{\mu}(u_\mu) = -\infty,
\]
which contradicts \eqref{eq:liminf}.
\end{proof}

 Now, we present the asymptotic behaviour of the \emph{ground-state} solutions. 

\begin{proposition}
\label{thm:limit_gs}
Assume that $(h_1)$ holds. Let $\lambda \in \mathbb R$ be fixed and $M>0$ be such that, for every $\mu\ge M$, $u_{\mu}$ is a solution of  \eqref{Plm} satisfying
    \begin{equation}
        \label{eq:cond-gs}
    I_{\mu}(u_{\mu})=\inf_{u\in C(\overline{\Omega})}I_{\mu}(u).
    \end{equation}
  Then, up to a subsequence,
    \begin{equation}
        \label{eq:conv-gs}
u_{\mu}
\to u_\infty
\quad \text{uniformly in } \overline{\Omega}
\quad \text{as } \mu\to+\infty,
\end{equation}
where $u_\infty$ solves the following maximum problem
\begin{equation}
\label{maxG}
\max_{u\in K_0}
\int_\Omega H(x,u)\,dx.
\end{equation}
In particular, under the further assumption $(h_2)$ on the nonlinearity $h$, the above conclusion holds for any family of ground-state solutions $\{u_{\mu}^{(l)}\}_{\mu\ge \widetilde\mu}$ as in Theorem \ref{thm:ground}. 
\end{proposition}

\begin{proof}
First of all, let us note that the maximum problem \eqref{maxG} is solvable by $(h_1)$ and by the compactness of $K_0$. Moreover, since $K_0$ is compact in
$C(\overline{\Omega})$, there exists $u_\infty \in K_0$
such that, up to a subsequence,
$$
u_\mu \to u_\infty
\quad \text{uniformly in } \overline{\Omega}.
$$
Assume by contradiction that $u_\infty$ does not solve \eqref{maxG}. Then there exists $u_*\in K_0$ such that 
\[
\int_\Omega H(x,u_*)dx > \int_\Omega H(x,u_\infty)dx. 
\]
Now, using that $u_\mu$ is characterized by \eqref{eq:cond-gs}, we have that
$$0
\le
I_{\mu}(u_*)
-
I_{\mu}(u_\mu)
=
\Psi(u_*)-\Psi(u_\mu)
-
\frac{\lambda}{2}\!\int_\Omega \bigl(u_*^2- u_\mu^2\bigr)\, dx
-
\mu\!\int_\Omega \bigl(H(x,u_*)-H(x,u_\mu)\bigr)\, dx.$$
Since $\varPsi$ is bounded in $K_0$, the first two terms are uniformly bounded in $\mu$, and similarly, the third by the uniform bound of $u\in K_0$.
On the other hand, by how we defined $u_*$, by $(h_1)$ and the uniform convergence of $\{u_\mu\}$, we have that
$$
\int_\Omega \bigl(H(x,u_*)-H(x,u_\mu)\bigr)\,dx
\to
\int_\Omega \big(H(x,u_*) - H(x,u_{\infty}) \big)\,dx > 0\quad\text{as $\mu\to+\infty$}.
$$
Altogether,
$$
I_{\mu}(u_*)
-
I_{\mu}(u_\mu)
=
O(1)
-
\mu
\int_\Omega \bigl(H(x,u_*)-H(x,u_\mu)\bigr)\, dx
\to -\infty \quad\text{as $\mu\to+\infty$,}
$$
which contradicts the minimality of $u_\mu$. Therefore, $u_\infty$ must solve \eqref{maxG}, concluding the proof.
\end{proof}

\begin{corollary}
Under the assumptions of Proposition \ref{thm:limit_gs}, if furthermore $(h_2')$ holds and $u_\mu\ge 0$ in $\Omega$ for every $\mu\ge M$, then $u_\infty = \operatorname{dist}(\cdot,\partial\Omega)$ and
\begin{equation}
\label{dist}
u_\mu \to \operatorname{dist}(\cdot,\partial\Omega) \quad\text{uniformly in $\overline{\Omega}$ as $\mu\to +\infty.$}
\end{equation}
\end{corollary}

\begin{proof} Since $u_\mu\ge 0$ in $\Omega$ for every $\mu\ge M$, by pointwise convergence also $u_\infty\ge 0$ in $\Omega$.
By the variational characterization of $u_\infty$ established in 
Proposition~\ref{thm:limit_gs} (precisely, in \eqref{maxG}), it suffices to show that 
$\operatorname{dist}(\cdot,\partial\Omega)$ is the unique maximizer of 
$u \mapsto \int_\Omega H(x,u)\,dx$ over $\{u\in K_0\,:\, u\ge 0 \text{ in }\Omega\}$. We first observe that $\operatorname{dist}(\cdot,\partial\Omega)\ge 0$ belongs to $K_0$ as it is 
$1$-Lipschitz and vanishes on $\partial\Omega$. Therefore, we need to prove that
\begin{equation}
\label{ineq-H}
    \int_\Omega H(x,u)\,dx \leq \int_\Omega H(x,\operatorname{dist}(\cdot,\partial\Omega))\,dx
    \quad \text{for every } 0\le u \in K_0,
\end{equation}
and that, if $0\le \bar u\in K_0$ is a solution of the maximum problem \eqref{maxG},
then $\bar u=\mathrm{dist}(\cdot,\partial\Omega)$.

By the proof of Proposition \ref{prop:propK0}-(i), we know that 
\[
u(x) \leq \operatorname{dist}(x,\partial\Omega) \quad 
    \text{for every } 0\le u \in K_0\text{ and every }x \in \Omega.
\]
Moreover, $(h_2')$ implies that, for a.e. $x \in \Omega$, $H(x,\cdot)$ is positive and strictly increasing in $(0,\infty)$, hence \eqref{ineq-H} follows. As for uniqueness, let $0\le \bar u\in K_0$ be another solution of \eqref{maxG}, then 
\[\int_\Omega H(x,\bar u) \,dx \le \int_\Omega H(x,\mathrm{dist}(\cdot,\partial\Omega)) \,dx \le \int_\Omega H(x,\bar u) \,dx\]
and so $H(x,\bar u(x))=H(x,\mathrm{dist}(x,\partial\Omega))$ for a.e. $x\in\Omega$. Finally, since $H(x,\cdot)$ is strictly monotone, this implies $\bar u(x)=\mathrm{dist}(x,\partial\Omega)$ for all $x\in \Omega$, thus concluding the proof.
\end{proof}

\begin{remark} When the domain is a ball $\mathcal B_R\subset\mathbb R^N$ and the problem is autonomous with a nonlinearity $h=h(u)$ of class $C^1$, all positive $C^2$-solutions of \eqref{Plm} are radially symmetric, cf. \cite[Appendix]{CoCoRi14} and \cite{GiNiNi79}. This would be the case, for instance, when $h(u)=|u|^{p-2}u$, for the ground-state solution and also for the mountain pass solution when $\lambda<\lambda_1$. In the next section, we will prove that, if $\mu$ is sufficiently large, there also exist multiple nodal radial solutions. It would be interesting to investigate whether the linking solution that we find when $\lambda\ge \lambda_1$ breaks the symmetry under suitable assumptions on $\mu$ or $R$. 
\end{remark}

\section{The radial case: a shooting approach}\label{sec3}

In this section, we focus on problem \eqref{Plm} on the radial domain $\Omega = \mathcal{B}_R(0)$, and with nonlinear term $h$ given by
$$
h(x,u) = q(|x|) g(u),
$$
where the following assumptions are made:
\begin{enumerate}
\item[$(g_1)$] $g: \mathbb{R} \to \mathbb{R}$ is locally Lipschitz continuous;

\item[$(g_2)$] there exist $\delta > 0$, $\gamma > 2$, and $0 < C_1 \leq C_2$ such that $C_1 |s|^\gamma \leq g(s) s \leq C_2 |s|^\gamma$ for every $ s \in (-\delta, \delta)$;

\item[$(q)$] $q: [0,R] \to \mathbb{R}$ is continuous and $ \displaystyle q_0:=\min_{r \in [0,R]} q(r) >0 $.
\end{enumerate} 
Incidentally, let us note that $(g_2)$ implies on one hand a superlinear behaviour for $h$ at $s= 0$ (that is $h(x,s)/s \to 0$ as $s \to 0$, cf. assumption $(h_3)$); on the other hand, together with $(q)$, it yields the sign-condition $h(x,s)s > 0$ for $0 \neq |s| \leq \delta$ (that is, assumption $(h_2)$). 

In this setting, we look for radial solutions $u(r) = u(|x|)$, thus leading to the problem 
\begin{equation}
\label{radial}
    \begin{cases}
        \begin{aligned}
           &- \left(r^{N-1}\frac{u^{\prime}}{\sqrt{1-(u^{\prime})^2}} \right)^\prime = r^{N-1} \left(\lambda u+ \mu q(r)g(u) \right), \quad &&r \in (0,R)  \, \\
           & u^\prime(0) = 0, \, u(R) = 0.
        \end{aligned}
    \end{cases}
\end{equation}
The section is organized as follows. In Section \ref{sec-3.1}, we state and prove an existence and multiplicity result for problem \eqref{radial}, with sharp information about the nodal behaviour of the solutions; we stress that in this result the parameters $\lambda$ and $\mu$ are required to be non-negative. In Section \ref{sec-bifurcation}, we discuss some bifurcation diagrams with respect to the parameter $\lambda$ and the parameter $\mu$.

For further convenience, we also introduce here the additional condition
\begin{enumerate}
\item[$(g_3)$] $g(s)s > 0$ for every $s \neq 0$ with $\vert s \vert \leq R$
\end{enumerate}
which will appear in the second part of the statement of Theorem \ref{mainresult}.

\subsection{Existence and multiplicity of nodal solutions}\label{sec-3.1}

Denoting by $\lambda_k^{\textnormal{rad}}$, with $k\geq 1$, the $k$-th radial eigenvalue of the Dirichlet Laplacian in $\mathcal{B}_R(0)$, and setting by convention $\lambda_0^{\textnormal{rad}} = 0$, the main result of the section reads as follows (of course, the choice $g(s) = \vert s \vert^{p-2} s$, with $p > 2$, and $q(r) \equiv 1$ provides the existence/multiplicity part of Theorem \ref{thm:radialintro} in the Introduction).

\begin{theorem}\label{mainresult}
Let us suppose that assumptions $(g_1)$, $(g_2)$ and $(q)$ hold true. Then, there exists a family $\{\mu_j^\star\}_{j= 1}^{+\infty}$ of functions $\mu_j^\star : [0, +\infty) \to [0, +\infty)$, satisfying, for every $j \geq 1,$
\begin{enumerate}
    \item[(i)] $\mu_j^\star(\lambda) = 0$ if $\lambda > \lambda_j^{\textnormal{rad}}$,
    \item[(ii)] $\mu_j^\star(\lambda) \leq \mu_{j+1}^\star(\lambda),$ for all $\lambda \geq0$,
\end{enumerate}
such that problem \eqref{radial} has at least $2j-k$ solutions when
$$
\lambda \in [\lambda_k^{\textnormal{rad}}, \lambda_{k+1}^{\textnormal{rad}}) \quad \text{and} \quad \mu > \mu_j^\star(\lambda), \qquad \text{for some } j\geq \max\{1, k\}.
$$
Precisely, these solutions can be labeled as $\{u_{\lambda,\mu,\ell}^{(s)}\}_{\ell =k+1}^j$, $\{u_{\lambda,\mu,\ell}^{(l)} \}_{\ell = 1}^j$ in such a way that
\begin{itemize}
    \item $0 < u_{\lambda,\mu,k+1}^{(s)}(0) < \dots < u_{\lambda,\mu,j}^{(s)}(0) < u_{\lambda,\mu,j}^{(l)}(0) < \dots < u_{\lambda,\mu,1}^{(l)}(0),$
    \item $u_{\lambda,\mu,\ell}^{(s)}$ has exactly $\ell -1$ zeros on $(0,R)$, for every $\ell = k+1, \dots, j$,
    \item $u_{\lambda,\mu,\ell}^{(l)}$ has exactly $\ell -1$ zeros on $(0,R)$, for every $\ell = 1, \dots, j$.
\end{itemize}
If, moreover, assumption $(g_3)$ holds, then all these solutions are strictly decreasing in the interval from $0$ to the first zero, and have exactly one local extremum between any two consecutive zeros.
\end{theorem}

\begin{remark}\label{rem_mu=0}
We point out that, as a byproduct of the arguments developed in this section, one can prove that, when $\lambda > \lambda_k^{\mathrm{rad}}$ for some integer $k\ge 1$, for every $\mu\ge 0$, problem \eqref{radial} has at least $k$ radial solutions $\{u_{\lambda,\mu,\ell}\}_{\ell=1}^k$ with $u_{\lambda,\mu,\ell}(0)>0$ and exactly $\ell-1$ zeros in $(0,R)$.
This result was already obtained by a bifurcation approach in \cite{DaWa17}; we can recover it by reasoning as in the proof of Theorem \ref{mainresult}, and using Lemmas \ref{lem:large} and \ref{lemmak} below, and the continuity of the map $\alpha \mapsto \theta_\alpha(R)$. Incidentally, note that weaker assumptions on $g$ and $q$ are needed in this argument ($q\in C([0,R])$ replacing $(q)$ and $g(s)/s\to 0$ as $s\to 0$ replacing $(g_2)$). 

With this remark, we can complete the multiplicity picture given by Theorem \ref{mainresult} with the further information
that the number of radial solutions of \eqref{radial} is at least $k$ when $\lambda> \lambda_k^{\mathrm{rad}}$ and $\mu = 0$ (note that, as $0=\mu^\star_1(\lambda)=\dots=\mu^\star_{k}(\lambda)=\dots=\mu^\star_j(\lambda)$ for $\lambda> \lambda_k^{\mathrm{rad}}$, the value $\mu = 0$ is actually the only one which is not covered by Theorem \ref{mainresult}). The resonant case 
$\lambda =  \lambda_k^{\mathrm{rad}}$, on the other hand, seems to be more delicate. We conjecture that, in this situation, a solution with $k-1$ zeros on $(0,R)$ does not exist when $\mu = 0$ (in the particular case $k=1$, this is actually proved in \cite{Mi21}) and that, even more, $\mu_k^\star(\lambda^\mathrm{rad}_k) > 0$ (namely, a solution with $k-1$ zeros on $(0,R)$ does not exist, in general, when $\mu$ is strictly positive but small).
\end{remark}

\begin{remark}
We remark that, with completely symmetric arguments, the existence of families of solutions $u_{\ell}^{(s)}$ and $u_{\ell}^{(l)}$ satisfying $u(0) < 0$ can be established, as well (when $g$ is odd, of course, it suffices to change sign to the solutions found in Theorem \ref{mainresult}). For brevity, however, we have preferred to state our result in this simpler form.
\end{remark}

The rest of this section will be devoted to the proof of Theorem \ref{mainresult}, which, as already mentioned in the Introduction, relies on a shooting technique. More precisely, in Section \ref{sec-radial1} we outline the shooting approach. Then, in Section \ref{sec-radial2} we prove some properties about the rotational behaviour of small and large solutions in the phase-plane. Finally, in Section \ref{sec-radial3} we focus on the behaviour of intermediate solutions and we conclude the proof of the theorem.

\subsubsection{The shooting approach}\label{sec-radial1}

Let us introduce
$$
\varphi(x) := \frac{x}{\sqrt{1-x^2}} \qquad \text{for all } x \in (-1,1)
$$ 
and its inverse 
\begin{equation}\label{inversebounded}
\varphi^{-1}(y) = \frac{y}{\sqrt{1+y^2}} \qquad \text{for all } y \in \mathbb{R};
\end{equation}
note that $$|\varphi^{-1}(y)| < 1 \qquad \text{for all } y \in \mathbb{R}.$$
Using this notation, of course, problem \eqref{radial} rewrites as 
\begin{equation}
  \label{eqequivalent}
    \begin{cases}
    -(r^{N-1} \varphi( u^\prime) )^\prime = r^{N-1}\left(\lambda u + \mu q(r) g(u)\right),\quad r\in(0,R),\\
    u^\prime(0) = 0, \quad u(R) = 0.
\end{cases}
\end{equation}
To prove Theorem \ref{mainresult}, we will look for solutions of \eqref{eqequivalent} using a shooting method. That is, we will be looking for $\alpha \neq 0$ such that the solution of the Cauchy problem
\begin{equation}
\label{42}
\begin{cases}
-(r^{N-1} \varphi( u^\prime) )^\prime = r^{N-1}\left(\lambda u + \mu q(r) g(u)\right),\quad r\in(0,R), \\
u'(0)=0,\quad u(0)=\alpha,
\end{cases}
\end{equation}
satisfies $u(R)=0$. 

Precisely, we consider the following Cauchy problem for the corresponding first-order planar system
\begin{equation}
\label{system}
\begin{cases}
u' = \displaystyle{\varphi^{-1}\!\left(\frac{v}{r^{N-1}}\right)}, \qquad\qquad\qquad & r \in (0, R),\smallskip\\
v' = -r^{N-1} ( \lambda u+ \mu q(r)g(u)), &r \in (0, R),\\
u(0) = \alpha, \quad
v(0) = 0.&
\end{cases}
\end{equation}
As shown in \cite{BoCoNo20}, despite the presence of the singular term $v/r^{N-1}$, the Cauchy problem \eqref{system} enjoys the standard well-posedness properties. In particular, the solutions are unique, globally defined, and continuously depend on the initial parameter $\alpha$.
These properties are summarized in the following lemma.

\begin{lemma}[{\cite[Lemma 2.1, Remark 2.2]{BoCoNo20}}]
\label{continuityalpha}
For any $\alpha \in \mathbb{R}$, problem \eqref{system} has a unique solution $(u_\alpha, v_\alpha)$ defined on $[0,R]$, which in particular satisfies $$ \lim_{r \to 0^+} \frac{v_\alpha (r)}{r^{N-1}} = 0,$$
that is, $v_\alpha(r) / r^{N-1}$ can always be extended continuously by $0$ at $r=0$.
Moreover, if $\alpha \to \bar{\alpha}$ , then
\begin{equation}
\label{unifalpha}
\begin{aligned}
    &u_{\alpha}(r) \to u_{\bar{\alpha}} (r) \quad &&\text{uniformly on } [0,R], \\
    &v_{\alpha}(r) \to v_{\bar{\alpha}} (r) \quad &&\text{uniformly on } [0,R],\\
    &\frac{v_\alpha(r)}{r^{N-1}} \to \frac{v_{\bar \alpha}(r)}{r^{N-1}} &&\text{uniformly on }[0,R].
\end{aligned}\end{equation}
\end{lemma}

\begin{remark}
\label{trivialsolution}
    Since the Cauchy problem \eqref{system} admits the trivial solution when $\alpha = 0$, and the solution is unique by Lemma~\ref{continuityalpha}, we get that $(u_0(r),v_0(r)) \equiv (0,0)$. 
\end{remark}

Let now $\alpha > 0$. Then, by the usual Cauchy-Lipschitz theory for ODEs, $(u_\alpha(r),v_\alpha(r)) \neq (0,0)$ for every $r \in [0,R]$ and so we can pass to scaled (and clockwise) polar coordinates around $(0,0)$, that is
\begin{equation}
\label{polar}
\begin{cases}
    \begin{aligned}
        &u_\alpha(r) = \beta_1 \rho_\alpha(r) \cos\theta_\alpha(r) \\
        &v_\alpha(r) = -\beta_2 \rho_\alpha(r) \sin \theta_\alpha(r),
    \end{aligned}
\end{cases}
\end{equation}
where $\rho_\alpha$ and $\theta_\alpha$ are continuous, $\rho_\alpha$ is positive, $\theta_\alpha(0) = 0$ and $\beta_1,\beta_2 > 0$ are scaling parameters.
The choice $\beta_1 = \beta_2 = 1$ gives rise, of course, to the standard polar coordinates and will be the preferred one. However, different choices will be convenient in some of the next computations (more precisely, in the proofs of Lemma \ref{lemmak} and Lemma \ref{claimellipse}): this is not a problem, since our final goal will be always that of proving estimates of the type
$$
\theta_\alpha(R) \; \substack{>\\=\\<} \; i \frac{\pi}{2}
$$
with $i$ an integer number, and it is easily seen that the validity of the above properties is independent of $\beta_1$ and $\beta_2$ (see {\cite[Proposition 2.2, Remark 2.2]{Bo11}} for a more formal discussion).

A standard computation yields in $(0,R)$
\begin{align}
    \label{eqtheta}
        \theta'_\alpha(r)& =\frac{1}{\beta_1 \beta_2 \rho_\alpha(r)^2} \left( \varphi^{-1}\left( \frac{v_\alpha(r)}{r^{N-1}}\right)v_\alpha(r) + r^{N-1} (\lambda u_\alpha(r)^2 + \mu  q(r) g(u_\alpha(r))u_\alpha(r)) \right)
     \\ & = \frac{1}{\beta_1 \beta_2} \left( \frac{\beta_2^2 \sin^2 \theta_\alpha(r)}{r^{N-1}\sqrt{1 + (v_\alpha(r)/r^{N-1})^2}}+ r^{N-1} \left( \lambda \beta_1^2 \cos^2 \theta_\alpha(r) + \frac{\mu q(r) g(u_\alpha(r))u_\alpha(r)}{\rho_\alpha(r)^2}\right) \right).\notag
\end{align}
Note that, by $(g_2)$, $g(s)s > 0$ for $0<|s| \leq \delta$ and so
\begin{equation}\label{crossing}
\theta_\alpha'(r) > 0, \quad 
\mbox{ whenever } \vert u_\alpha(r) \vert \leq \delta.
\end{equation}
In particular, $\theta_\alpha'(r) > 0$ if $u_\alpha(r) = 0$, meaning that the $v$-axis can be crossed only in the clockwise sense.

With this in mind, our goal will be that of finding suitable values of $\alpha$ such that
\begin{equation}\label{nodaltheta}
\theta_\alpha(R) = \frac{\pi}{2} + (\ell-1) \pi, \quad \mbox{ for some integer } \ell \geq 1.
\end{equation}
Indeed, by \eqref{polar} this immediately yields $u_\alpha(R) = 0$, and so $u_\alpha$ solves \eqref{eqequivalent}. Moreover, if \eqref{nodaltheta} is satisfied, $u_\alpha$ has exactly $\ell-1$ zeros on $(0,R)$: as discussed above the $v$-axis can be crossed only in the clockwise sense, and so from \eqref{nodaltheta} we infer that the number of crossing of this axis in $(0,R)$ (that is, the number of zeros of $u_\alpha$) is exactly $\ell-1$. Finally, since (due to the bound $|u'| < 1$) the solutions are a-priori bounded by $R$, if $(g_3)$ holds we get that the angular coordinate $\theta_\alpha$ is always strictly increasing, and so between any two consecutive zeros of $u_\alpha$ there is exactly one zero of $v_\alpha$ (and thus of $u_\alpha'$), and between any two consecutive zeros of $u'_\alpha$ there is exactly one zero of $u_\alpha$.

\subsubsection{The behaviour of small and large solutions}\label{sec-radial2}

Our first, very simple, lemma concerns the angular behaviour of large solutions (namely, those with $\alpha > R$) of \eqref{system}. This is an immediate consequence of the bound $|u'|<1$ of solutions which, by Proposition \ref{prop:propK0}-(i), gives the bound $\alpha=u(0)\le \tfrac12\mathrm{diam}(\mathcal B_R(0))=R$ on the initial datum. More precisely, we have the following result.

\begin{lemma}\label{lem:large}
Let $\lambda,\mu \geq 0$. Then
 \begin{equation}
        \theta_\alpha(R) <  \frac{\pi}{2}, \quad \text{ if } \alpha \geq R+1.
    \end{equation}
\end{lemma}

\begin{proof}

Due to the bound 
\[
\vert u'(r) \vert = \left\vert \varphi^{-1}\left( \frac{v(r)}{r^{N-1}}\right)\right\vert < 1, \quad \text{ for every } r \in [0,R],
\]
cf. \eqref{inversebounded}, 
we have, for $\alpha \geq R+1$,
\begin{equation}\label{stima-apriori}
u_\alpha(r) = \alpha + \int_0^r u'(s)ds \geq R+1 - R = 1, \quad \text{ for every } r \in [0,R].
\end{equation}
From this, the conclusion easily follows.
\end{proof}

We now describe the angular behaviour of small solutions (namely, those with $\alpha$ close to $0$)  of \eqref{system}. 

\begin{lemma}
\label{lemmak+1}
    Let $\lambda \in [0,\lambda^{\textnormal{rad}}_{k+1})$ for some $k \geq 0$, and let $\mu \geq 0$. Then there exists $\alpha_\star = \alpha_\star(\lambda,\mu)>0$ such that
    \begin{equation}
        \theta_\alpha(R) < \frac{\pi}{2} + k \pi, \quad \text{ if } \alpha \in (0,\alpha_\star).
    \end{equation}
\end{lemma}
\begin{proof}
Since $\lambda < \lambda^{\textnormal{rad}}_{k+1}$, there exists $\varepsilon > 0$ such that
\begin{equation}
\label{eq:lambda_upper}
    \lambda < \lambda^{\textnormal{rad}}_{k+1} - \varepsilon.
\end{equation}
In view of $(g_2)$ and $(q)$, we can thus find $\tilde \delta \in (0,\delta)$ such that
\begin{equation}
\label{delta}
    0 \leq \mu q(r) g(u) u \leq\frac{\varepsilon}{2}\, u^2
    \qquad\text{ for all }  |u| \leq \tilde\delta \text{ and } \, r \in [0,R].
\end{equation}
By Lemma~\ref{continuityalpha} and Remark \ref{trivialsolution}, the solution
$(u_\alpha, v_\alpha)$ converges uniformly to $(0,0)$ on $[0,R]$ as $\alpha \to 0^+$.
Therefore, there exists $\alpha_\star > 0$ such that, for every $\alpha \in (0, \alpha_\star)$,
\begin{equation}
\label{eq:bound_upper}
    |u_\alpha(r)| \leq \tilde\delta, \qquad\text{ for all } r \in [0,R].
\end{equation}
    We now estimate the terms in the angular equation~\eqref{eqtheta},
setting $\beta_1=\beta_2=1$ in~\eqref{polar}.

    \begin{itemize}
        \item \textbf{Estimate of the $\sin^2\theta_\alpha$ term.} Since $1 + (v_\alpha/r^{N-1})^2 \geq 1$, we have
\begin{equation}
\label{eq:sin_upper}
\frac{\sin^2\theta_\alpha}{r^{N-1}\sqrt{1+(v_\alpha/r^{N-1})^2}}
    \leq \frac{\sin^2\theta_\alpha}{r^{N-1}}.
\end{equation}
\item \textbf{Estimate of the $\cos^2\theta_\alpha$ term.}
Conditions ~\eqref{delta} and \eqref{eq:bound_upper} give, for $\alpha \in (0,\alpha_\star)$,
$$
    \frac{r^{N-1}\,\mu\,q(r)\,g(u_\alpha)u_\alpha}{\rho_\alpha^2}
    \leq \frac{\varepsilon}{2}\,r^{N-1}\cos^2\theta_\alpha.
$$
Therefore, using ~\eqref{eq:lambda_upper} yields
\begin{equation}
\label{eq:cos_upper}
    r^{N-1}\lambda\cos^2\theta_\alpha
    + \frac{r^{N-1}\,\mu\,q(r)\,g(u_\alpha)u_\alpha}{\rho_\alpha^2}
    \leq r^{N-1}\,\overline\lambda\,\cos^2\theta_\alpha,
\end{equation}
where $\overline\lambda := \lambda^{\textnormal{rad}}_{k+1} - \varepsilon/2 < \lambda^{\textnormal{rad}}_{k+1}$.
    \end{itemize}
Collecting~\eqref{eq:sin_upper} and~\eqref{eq:cos_upper} into~\eqref{eqtheta},
we obtain, for every $\alpha \in (0,\alpha_\star)$,
$$
    \theta_\alpha'
    \leq \frac{\sin^2\theta_\alpha}{r^{N-1}}
    + r^{N-1}\,\overline\lambda\,\cos^2\theta_\alpha.
$$
     Then, considering $\vartheta_{\lambda_{k+1}^{\textnormal{rad}}}$ as the solution of
     $$\vartheta_{\lambda_{k+1}^{\textnormal{rad}}}' = \frac{\sin ^2 \vartheta_{\lambda_{k+1}^{\textnormal{rad}}}}{r^{N-1}} +  r^{N-1} \,\lambda_{k+1}^{\textnormal{rad}} \cos^2 \vartheta_{\lambda_{k+1}^{\textnormal{rad}}} \quad \text{in } (0,R),$$
     and following the lines of \cite{BoCoNo20}, through the Comparison Theorem for ODEs, we can conclude that $$\theta_\alpha(R) < \vartheta_{\lambda_{k+1}^{\text{rad}}}(R) = \frac{\pi}{2} + k\pi,$$ 
     where the last equality comes from the observation that $\vartheta_{\lambda_{k+1}^{\text{rad}}}$ is the angular part of the $(k+1)$-th radial eigenfunction of the Laplacian and from the known properties of the eigenvalue problem of the Laplacian in the radial case.
\end{proof}

\begin{lemma}
\label{lemmak}
    Let $\lambda > \lambda^{\textnormal{rad}}_k$ with $k\geq 1$. Then, for every $\mu\geq0$ there exists $\alpha_{\mu} >0$ such that \begin{equation}
        \theta_{\alpha}(R) > \frac{\pi}{2} + (k-1)\pi,\qquad \text{if } \alpha \in(0, \alpha_{\mu}). 
    \end{equation}
\end{lemma}
\begin{proof}
 Since $\lambda > \lambda^{\textnormal{rad}}_k$, there exists $\bar\lambda$ such that 
\[
\lambda^{\textnormal{rad}}_k < \bar\lambda < \lambda
\]
and consequently, we can choose
$\varepsilon > 0$ such that
\begin{equation}
\label{eq:lambda-eps}
    \bar\lambda\sqrt{1+\varepsilon^2} <  \lambda - \varepsilon.
\end{equation}
In view of $(g_2)$ and $(q)$, 
\begin{equation} \label{eq:g2_bound_lower}
\lambda s^2 + \mu q(r) g(s) s \ge (\lambda-\varepsilon) s^2 \ge \bar\lambda\sqrt{1+\varepsilon^2} s^2, \qquad\text{for all } |s| < \delta \text{ and } r\in [0,R].\end{equation}
By Lemma~\ref{continuityalpha} and Remark \ref{trivialsolution}, there exists $\alpha_\mu > 0$ such that, for every $\alpha \in (0,\alpha_\mu)$,
\begin{equation}
\label{eq:bound_lower}
    |u_\alpha(r)| < \delta
    \quad\text{and}\quad
    \frac{|v_\alpha(r)|}{r^{N-1}} \leq \varepsilon,
    \qquad  \text{for all } r \in [0,R].
\end{equation}
Now, we set
\begin{equation} \label{b1b2}
     \beta_1 := 1,\quad  \beta_2 := \sqrt{1 + \varepsilon^2}
\end{equation}
in \eqref{polar} and estimate the terms in the angular equation \eqref{eqtheta} under
conditions \eqref{eq:g2_bound_lower} and \eqref{eq:bound_lower}, to get, for every $\alpha\in (0,\alpha_\mu)$,
$$
    \theta_\alpha'(r)
    \geq \frac{\sin^2\theta_\alpha(r)}{r^{N-1}}
    + r^{N-1}\,\overline\lambda\,\cos^2\theta_\alpha(r),\qquad\text{for all }r\in[0,R].
$$
Since $\overline\lambda > \lambda^{\textnormal{rad}}_k$, once again reasoning
as in \cite{BoCoNo20}, the Comparison Theorem for ODEs and the spectral
properties of the radial Laplacian give
$$
    \theta_\alpha(R)
    > \vartheta_{\lambda^{\textnormal{rad}}_k}(R) = \frac{\pi}{2} + (k-1)\pi,\qquad\text{for every }\alpha\in (0,\alpha_\mu),
$$
where $\vartheta_{\lambda^{\textnormal{rad}}_k}$ denotes the angular
component of the $k$-th radial eigenfunction of the Laplacian.

\end{proof}

\subsubsection{The behaviour of intermediate solutions and conclusion of the proof of Theorem \ref{mainresult}}\label{sec-radial3}

For every $\lambda \geq 0$ and every integer $j \geq 1$, we now define
$$
\mu_j^\star(\lambda)
\;:=\;
\inf \Bigl\{ \bar{\mu} > 0 \;\colon\;
\forall\, \mu > \bar{\mu},\;
\exists\, \alpha > 0 \ \text{such that} \ 
\theta_\alpha(R) > \tfrac{\pi}{2} + (j-1)\pi
\Bigr\}.
$$
Note that the above quantity actually depends on $\lambda$, since the angular coordinate $\theta_\alpha$ itself depends on $\lambda$. Of course, condition (ii) appearing in Theorem \ref{mainresult} is automatically satisfied. 
Moreover, by
Lemma \ref{lemmak} with $k = j \geq 1$, it follows that 
$$
\mu_j^\star(\lambda) = 0 \quad \text{if } \lambda > \lambda_j^{\mathrm{rad}}.
$$
As a consequence, condition (i) in Theorem \ref{mainresult} also holds true.

\medskip

In principle it could be that, for some $j \geq 1$, there exists $\lambda \in [0,\lambda_j^{\textnormal{rad}}]$ for which $\mu_j^\star(\lambda) = +\infty$. \emph{Assuming however that this is not the case} (a fact that will be shown later), the proof of Theorem \ref{mainresult} can be easily given.

\medskip

Indeed, let
$$
\lambda \in [\lambda^{\textnormal{rad}}_k, \lambda^{\textnormal{rad}}_{k+1}) \quad \text{and} \quad \mu > \mu_j^\star(\lambda), \qquad \text{for some } j\geq \max\{1,k\}.
$$
By the definition of $\mu_j^\star$, there exists $\alpha^\star = \alpha^\star(\lambda,\mu)>0$ such that
\begin{equation}\label{eq:alpha-star-final}
\theta_{\alpha^\star}(R)>\frac{\pi}{2}+(j-1)\pi .
\end{equation}
On the other hand, by Lemma \ref{lemmak+1}, there exists $\alpha_\star = \alpha_\star(\lambda,\mu)>0$ such that
\begin{equation}\label{eq:small-alpha-upper-final}
\theta_\alpha(R)<\frac{\pi}{2}+k\pi,
\qquad \text{for all } \alpha\in(0,\alpha_\star).
\end{equation}
Without loss of generality, we can take $\alpha_\star$ so small that $\alpha_\star < \alpha^\star$. Finally, Lemma \ref{lem:large} gives
$$
\theta_{\alpha}(R) < \frac{\pi}{2},
\qquad \text{for all } \alpha \geq R+1,
$$
and so $\alpha^\star < R+1$.

Since the map $\alpha\mapsto\theta_\alpha(R)$ is continuous, a first application of the Intermediate Value
Theorem shows that, for every $\ell=1,\dots,j$, there exists $\alpha_{l,\ell}\in(\alpha^\star,R+1)$ such that
$$
\theta_{\alpha_{l,\ell}}(R)=\frac{\pi}{2}+(\ell-1)\pi .
$$
As discussed at the end of Section \ref{sec-radial1}, this means that we have found $j$ different solutions of \eqref{eqequivalent}, and that $\ell-1$ (with $\ell = 1,\dots,j$) is the number of zeros of such solutions in $(0,R)$: these are the solutions which in Theorem \ref{mainresult} are denoted by $u_{\lambda,\mu,\ell}^{(l)}$. On the other hand, when $j > k$, a second application of the Intermediate Value
Theorem gives, for every $\ell=k+1,\dots,j$, the existence of $\alpha_{s,\ell}\in(0,\alpha_\star)$ such that
$$
\theta_{\alpha_{s,\ell}}(R)=\frac{\pi}{2}+(\ell-1)\pi.
$$
Similarly as before, this provides the $j-k$ solutions (denoted by $u_{\lambda,\mu,\ell}^{(s)}$ in Theorem \ref{mainresult}) with $\ell-1$ zeros on $(0,R)$ for $\ell=k+1,\dots,j$.

\begin{figure}[ht]
    \centering
    \includegraphics[width=0.85\linewidth]{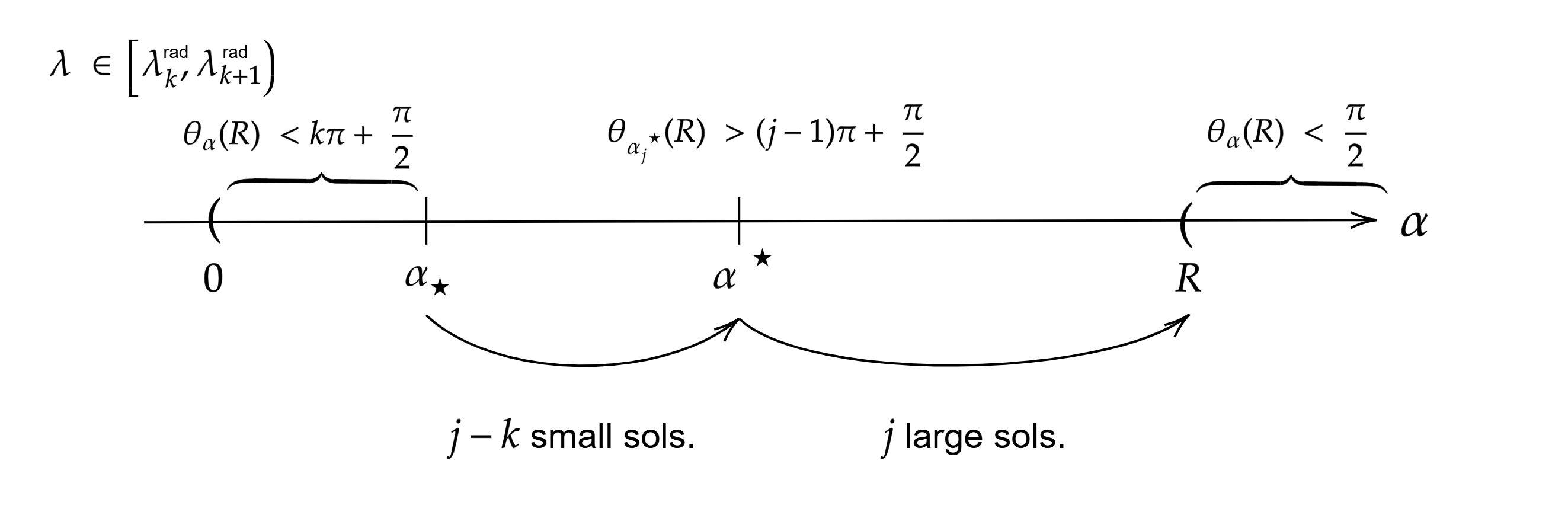}
    \caption{Representation of the different values $\alpha \mapsto \theta_\alpha(R)$ for $\mu \geq \mu_j^\star$, $j \geq k$.}
    \label{fig:}
\end{figure}

To conclude the proof, we must yet prove that
\begin{equation}\label{mustarfinito}
\mu_j^\star(\lambda) < + \infty \quad \text{for all } j \geq 1, \text{ and } \lambda \geq0.
\end{equation}
This will be a rather direct consequence of the following lemma.

\begin{lemma}
\label{claimellipse}
Let $\lambda \geq 0$. For every $r_0 \in (0,R)$, and for every integer $i \geq 1$, there exist $\tilde\mu_i(\lambda, r_0)\geq 1$ and $\Gamma_i(\lambda, r_0) \in (0, \delta)$ such that, for every $\mu > \tilde\mu_i(\lambda, r_0)$, every solution $(u(r),v(r))$ of \eqref{system} such that $\mu^{2/\gamma }u(r_0)^2 + \ v(r_0)^2 = \Gamma_i(\lambda, r_0)^2$ fulfills
    \begin{equation}
    \label{ellipsis}
    \theta(R)-\theta(r_0) > (i+1)\pi \end{equation}
where $\gamma$ is as in $(g_2)$ and the angular coordinate $\theta$ is defined by
$$(u(r), v(r)) = (\mu^{-1/\gamma}\rho(r) \cos \theta(r), - \rho(r)\sin \theta(r)).$$
\end{lemma}

Let us first show how from this lemma we can obtain \eqref{mustarfinito}. So, let $\lambda \geq 0$ and $j \geq 1$ be fixed, and freely choose $r_0 \in (0,R)$. Without loss of generality, we also assume $\delta \in (0,1)$.
Then, define $\hat \mu_j(\lambda) := \tilde\mu_j(\lambda, r_0)$ as given by Lemma \ref{claimellipse}, with $i=j$. We claim that
\begin{equation}\label{mustarfinito2} \mu_j^\star(\lambda) \leq \hat \mu_j(\lambda),
\end{equation}
which finally implies that $\mu_j^\star(\lambda)$ is finite (that is, \eqref{mustarfinito}).
So, let us fix $\mu > \hat \mu_j(\lambda)$.
Since $(u_0(r),v_0(r)) \equiv (0,0)$
and, as in \eqref{stima-apriori}, $u_{R+1}(r) \geq 1$ for every $r \in [0,R]$, and recalling that $\Gamma_j(\lambda,r_0) < \delta < 1 < \mu$, by an elementary continuity argument there exists $\hat\alpha_j=\hat\alpha_j(\lambda,\mu) \in (0,R+1)$ such that the solution $(u_{\hat\alpha_j},v_{\hat\alpha_j})$ satisfies
\begin{equation}
\label{ellipse}
\mu^{2/\gamma}u_{\hat\alpha_j}(r_0)^2+v_{\hat\alpha_j}(r_0)^2=\Gamma_j(\lambda,r_0)^2.
\end{equation}
Let us now consider the associated angular coordinate $\theta_{\hat\alpha_j}$, choosing $\beta_1 := \mu^{-1/\gamma}$ and $\beta_2 := 1$ as scaling parameters in \eqref{polar}, in such a way that it coincides with the angular coordinate used in Lemma \ref{claimellipse}.
Of course, we have
\[
\theta_{\hat\alpha_j}(R) = 
\theta_{\hat\alpha_j}(R) -\theta_{\hat\alpha_j}(0) =
\left(\theta_{\hat\alpha_j}(R) - \theta_{\hat\alpha_j}(r_0) \right) + \left(\theta_{\hat\alpha_j}(r_0) - \theta_{\hat\alpha_j}(0)\right).
\]
Now, by \eqref{ellipsis} the first term in the above sum is greater than $(j+1)\pi$. Moreover, due to the fact that the $v$-axis can be crossed only in the clockwise sense (as discussed just below \eqref{crossing}), the second one is greater than $-\pi$, cf. \cite[Lemma 3.3]{BoGa19}. Therefore we have 
\[
\theta_{\hat\alpha_j}(R) > (j+1) \pi - \pi = j \pi > \frac{\pi}{2} + (j-1)\pi.
\]
By the definition of $\mu_j^\star(\lambda)$, \eqref{mustarfinito2} is thus proved, concluding the proof.

\begin{proof}[Proof of Lemma 
\ref{claimellipse}]
We want to apply
    Proposition \ref{Prop21} to an appropriate problem. To this end, we introduce a rescaled variable $t$ defined by 
    \[
    r = \frac{t}{\mu^{1/\gamma}}.
    \] 
    Note that
    \[
    r \in [r_0,R] \quad  \Longleftrightarrow  \quad t \in I_\mu :=[\mu^{1/\gamma}r_0, \mu^{1/\gamma} R].
    \]
    Moreover, we set
$$
(x(t),y(t)) := \left( \mu^{1/\gamma}\, u\!\left(\frac{t}{\mu^{1/\gamma}} \right), \; v\!\left(\frac{t}{\mu^{1/\gamma}} \right) \right).
$$
With this transformation, the pair $(x(t),y(t))$ satisfies the system
$$
\begin{cases}
\begin{aligned}
x' &= \varphi^{-1}\left( 
\left(\frac{\mu^{1/\gamma}}{t} \right)^{N-1} y\right)=: X_\mu(t,y), \\
y' &= - \left(\frac{t}{\mu^{1/\gamma}}\right)^{N-1} 
\left( \frac{\lambda}{\mu^{2/\gamma}}x  + \mu^{(\gamma-1)/\gamma} q \left(\frac{t}{\mu^{1/\gamma}}\right)g\left(\frac{x}{\mu^{1/\gamma}}\right) \right)
=:\, -Y_\mu(t,x),
\end{aligned}
\end{cases}
$$
with $t \in I_\mu$. 
Now, for every $\mu \geq 1$, $t \in I_\mu$, and $y \in \mathbb{R}$, by the sign and monotonicity of $\varphi^{-1}$, it holds that
\[
\varphi^{-1} \left( \frac{y}{R^{N-1}}\right) y\leq X_\mu(t,y)y \leq  \varphi^{-1} \left( \frac{y}{r_0^{N-1}}\right) y. 
\]
On the other hand, for every $\mu \geq 1$, $t \in I_\mu$ and $x \in (-\delta,\delta)$, using assumptions  $(g_2)$ and $(q)$ we get
\[
r_0^{N-1} q_0 C_1 |x|^{\gamma}\leq Y_\mu(t,x)x \leq R^{N-1} \left( \lambda x^2 + \|q\|_{L^\infty(0,R)} C_2 |x|^{\gamma}\right).
\]
The above estimates allow us to enter the setting of Proposition \ref{Prop21} with the choices
\[
a_1(y) := \varphi^{-1} \left( \frac{y}{R^{N-1}}\right), \quad b_1(y) := \varphi^{-1} \left( \frac{y}{r_0^{N-1}}\right)
\]
and
\[
b_2(x) := r_0^{N-1} q_0 C_1 |x|^{\gamma-2} x, \quad a_2(x) := R^{N-1} \left( \lambda x + \|q\|_{L^\infty(0,R)} C_2 |x|^{\gamma-2}x\right).
\]
Accordingly, Proposition \ref{Prop21} can be applied (for $k = i+1$)
yielding two positive constants $\rho^*_{i+1}(\lambda,r_0)=:\Gamma_i(\lambda,r_0) \in (0,\delta)$
and $\tau^*_{i+1}(\lambda,r_0) > 0$ such that
\[
x \left(\mu^{1/\gamma}r_0\right)^2 + y \left(\mu^{1/\gamma}r_0\right)^2 = \Gamma_i(\lambda,r_0)^2
\]
implies
\[
\vartheta(\mu^{1/\gamma}R) - \vartheta(\mu^{1/\gamma}r_0) > (i+1)\pi
\]
provided $ \mu \geq 1$ and $|I_\mu| =\mu^{1/\gamma}(R - r_0) > \tau^*_{i+1}(\lambda,r_0)$. Here $\vartheta$ denotes the standard (clockwise) angular coordinate of the path $(x(t),y(t))$. 
Setting 
\[
\tilde\mu_i(\lambda, r_0) := \max \left\{1, \left(\frac{\tau^*_{i+1}(\lambda,r_0)}{R-r_0}\right)^{\gamma}\right\}
\]
and going back to the original variables $(u(r),v(r))$, we get the conclusion.
\end{proof}

\subsection{Bifurcation diagrams}\label{sec-bifurcation}

In this section, we present conjectural bifurcation diagrams, both with respect to the parameters $\mu$ and $\lambda$, for problem \eqref{radial}. We stress that the proposed bifurcation structure is neither rigorously justified (indeed, our shooting approach does not even provide, in principle, the existence of connected sets of solutions $(\lambda,\mu,u_{\lambda,\mu})$) nor fully validated by numerical computations. Nevertheless, we believe that the depicted diagrams accurately capture the actual bifurcation scenario. This conviction is supported by several independent pieces of evidence: they are fully consistent with the multiplicity scheme established in Theorem~\ref{mainresult}; they agree with the analytical and numerical results available for closely related problems (see, in particular, \cite{BoGa19,BoCoNo20,DaWa17}), and they are corroborated by the partial numerical simulations performed during this work.

For these reasons, we regard the proposed diagrams as a useful qualitative framework for interpreting the theoretical multiplicity results obtained in the previous sections. A rigorous theoretical analysis, together with a systematic numerical validation of these bifurcation diagrams, will be the subject of future investigation.

We now describe in more detail the diagrams appearing in Figure \ref{figmu} and Figure \ref{figlambda}. In all of them, we have depicted three branches of solutions, precisely positive solutions (in blue), one-node solutions (in purple), and two-node solutions (in orange). For simplicity, from now on, the radial eigenvalue 
$\lambda_k^{\textnormal{rad}}$ is simply denoted as $\lambda_k$.

\smallbreak
\noindent 
\textbf{Bifurcation with respect to $\mu$ - Figure \ref{figmu}}. The three diagrams in Figure~\ref{figmu} illustrate the conjectured bifurcation structure with respect to the parameter $\mu$, for different ranges of the parameter $\lambda$.

The first diagram corresponds to $\lambda \in [0,\lambda_1)$. In this regime, each branch is $\subset$-shaped, with a turning point located at $\mu_k^*(\lambda)$, $k=1,2,3$. Thus, increasing $\mu$, solutions appear in pairs, consisting of a small and a large solution with the same nodal properties, corresponding to the lower and the upper portions of each branch, respectively (when $\lambda = 0$, this is actually the situation already investigated in \cite{BoGa19}). 
We emphasize that, strictly speaking, this is not a genuine bifurcation phenomenon, since the branches accumulate on the trivial line $\{u=0\}$ only as $\mu\to+\infty$.

As the parameter $\lambda$ increases, these $\subset$-shaped branches progressively shift to the left. In this perspective, the second diagram corresponds to $\lambda\in(\lambda_1,\lambda_2)$. Here $\mu_1^\star(\lambda)=0$, meaning that the upper portion of the branch of positive solutions reaches the axis $\{\mu=0\}$, while its lower portion degenerates into the trivial line $\{u=0\}$. By contrast, the qualitative behaviour of the branches corresponding to one-node and two-node solutions remains essentially unchanged. Consequently, there exists one positive solution for $\mu\leq\mu_2^\star(\lambda)$, three solutions (one positive solution together with a pair of one-node solutions) for $\mu\in(\mu_2^\star(\lambda),\mu_3^\star(\lambda)]$, five solutions for $\mu>\mu_3^\star(\lambda)$, and so on if additional branches are taken into account.

Finally, the third diagram depicts the case $\lambda\in(\lambda_2,\lambda_3)$. Here, $\mu_1^\star(\lambda)=\mu_2^\star(\lambda)=0$, so that both the branch of positive solutions and the branch of one-node solutions reach the axis $\{\mu=0\}$, whereas the branch of two-node solutions still has a $\subset$-shape. As a consequence, there are two solutions (one positive and one with one node) for $\mu\leq\mu_3^\star(\lambda)$, while for $\mu>\mu_3^\star(\lambda)$ an additional pair of two-node solutions appears. The same pattern continues as $\mu$ increases, with further pairs of higher-nodal solutions arising whenever additional branches are considered.

\smallbreak
\noindent
\textbf{Bifurcation with respect to $\lambda$ - Figure \ref{figlambda}}. 
Figure \ref{figlambda} shows four different bifurcation diagrams with respect to the parameter $\lambda$, for different values of the parameter $\mu$. 

Here, the (not proved in this paper, but highly expected) proposed bifurcation scenario is that, for every $\mu \geq 0$, the branch of $(k-1)$-node solutions bifurcates from the eigenvalue $\lambda_k$. In the first diagram, the case $\mu = 0$ is depicted: in this setting, each branch immediately moves to the right, cf. Remark \ref{rem_mu=0}. Increasing $\mu$, the branches are pushed to the left (but still anchored to their bifurcation point $(\lambda_k,0)$). In particular, in the third diagram, the branch of positive solutions reached the axis $\{\lambda = 0\}$: accordingly, a pair of positive solutions is obtained for $\lambda \in [0,\lambda_1)$. For $\lambda \in [\lambda_1,\lambda_2)$, on the other hand, only one positive solution is obtained: in this case, however, other pairs of higher node solutions may be provided by the other branches. For instance, in the situation depicted in the third diagram, a pair of one-node solutions exists, while, for larger values of $\mu$, like in the fourth diagram, an additional pair of two-node solutions appears as well. Note that, in this last diagram, also the branch of one-node solutions has reached the axis $\{\lambda = 0\}$, so giving the existence of a pair of one-node solutions also in the interval $[0,\lambda_1)$. Of course, considering more branches and increasing $\mu$ leads to the full multiplicity scheme of Theorem \ref{mainresult}.

\begin{figure}[ht]
\centering

\includegraphics[width=0.48\textwidth]{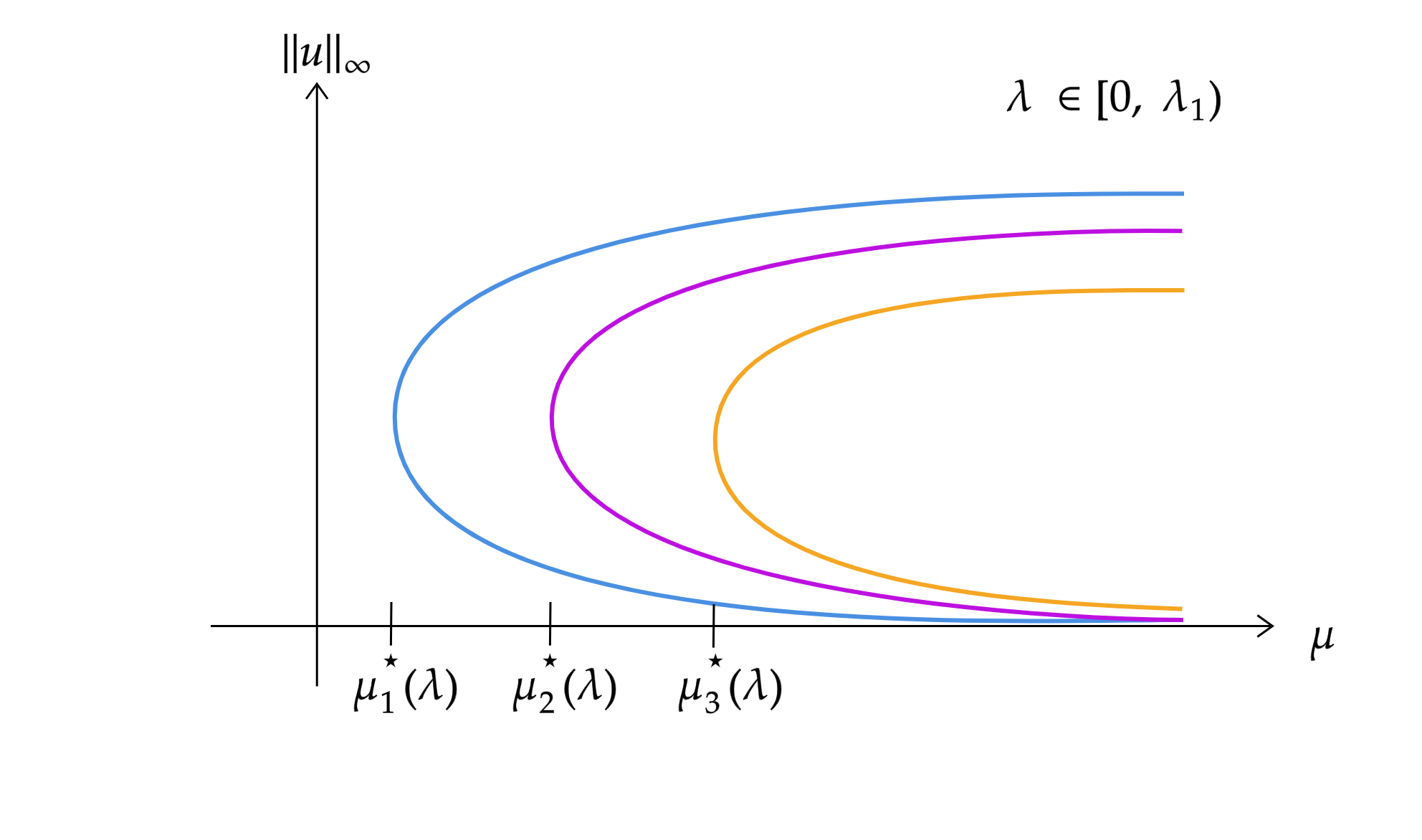}
\hspace{0.02\textwidth}
\includegraphics[width=0.48\textwidth]{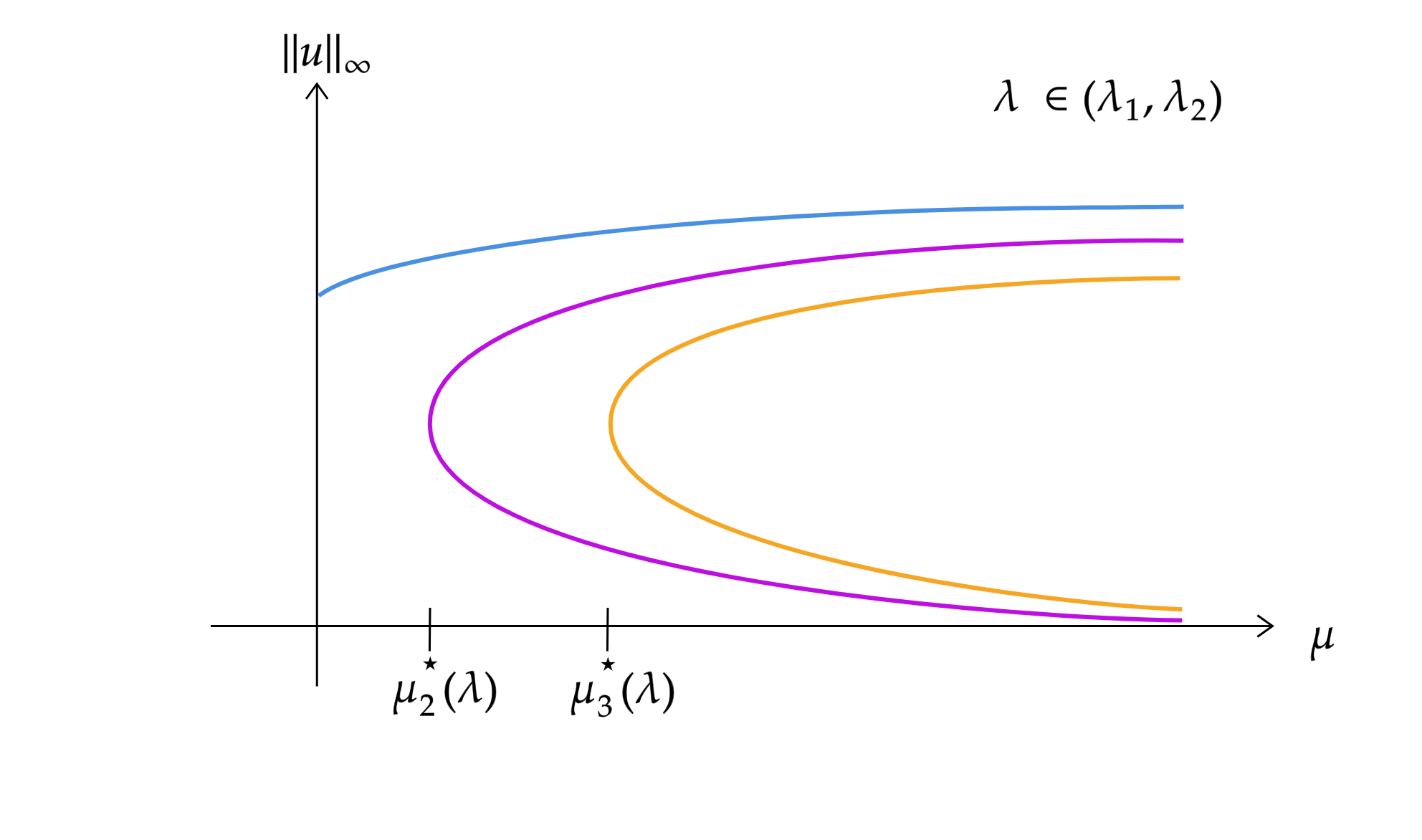}

\includegraphics[width=0.48\textwidth]{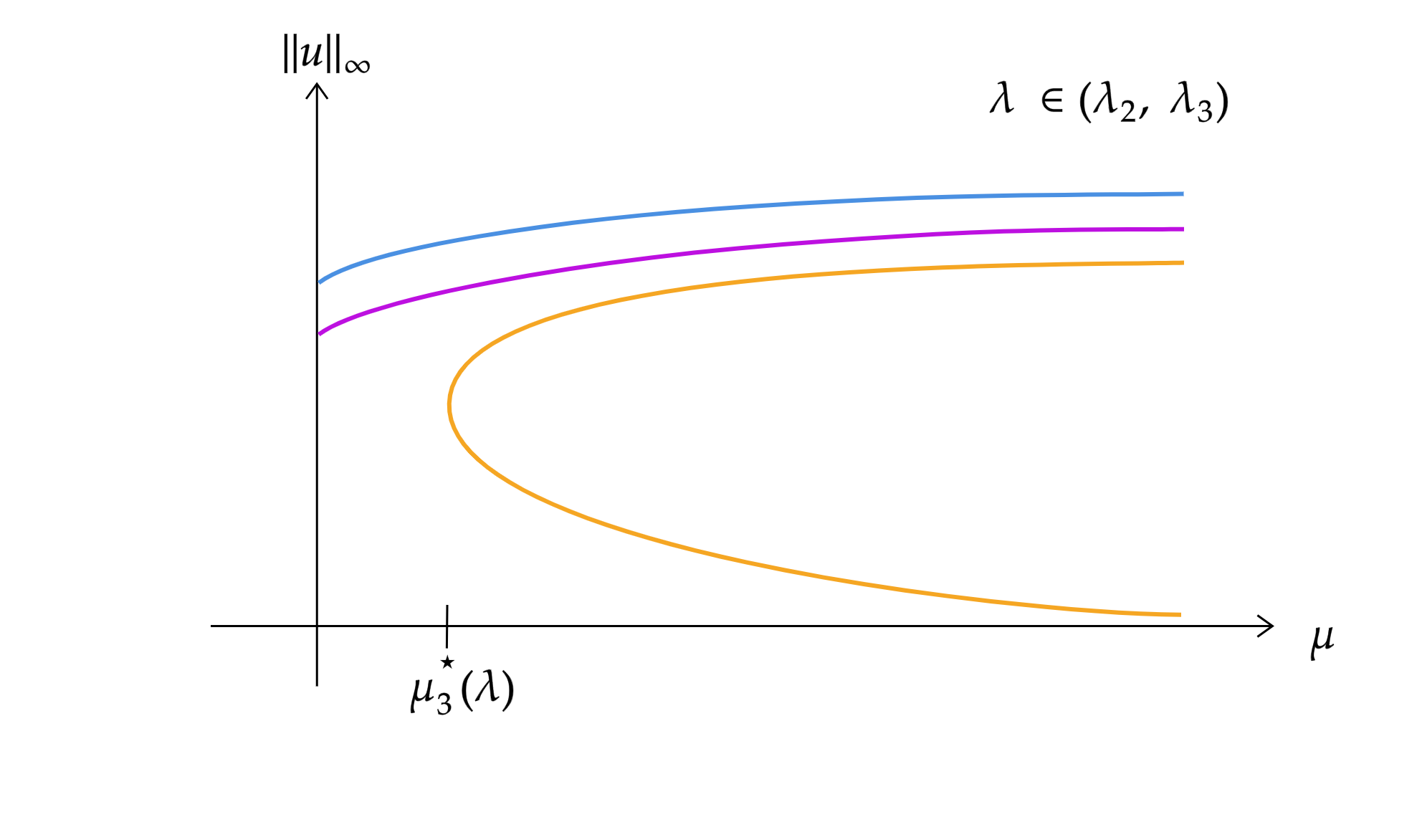}

\caption{Bifurcation diagrams with respect to $\mu$, for different values of $\lambda$.}
\label{figmu}
\end{figure}

\begin{figure}[ht]
\centering

\includegraphics[width=0.48\textwidth]{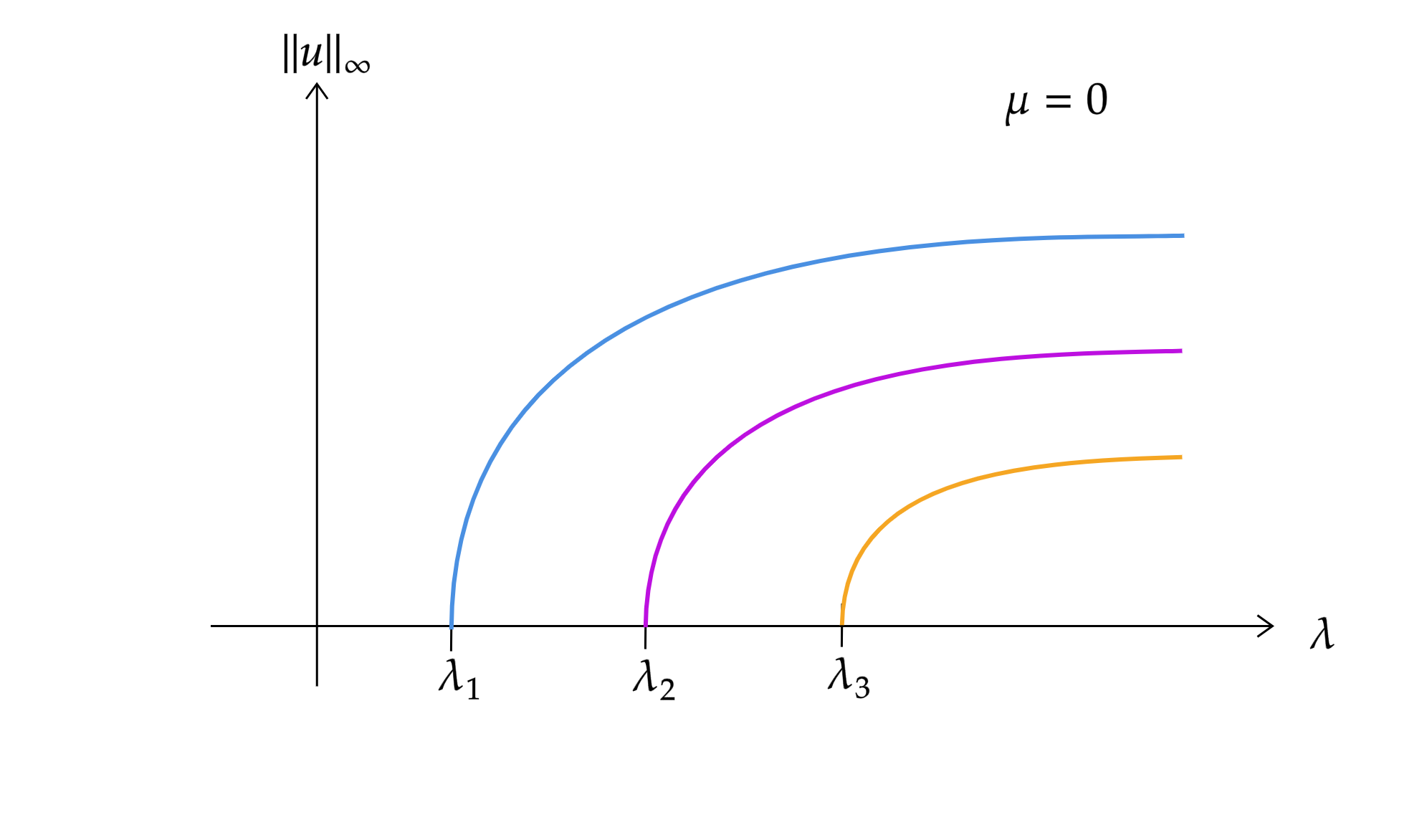}
\hspace{0.02\textwidth}
\includegraphics[width=0.48\textwidth]{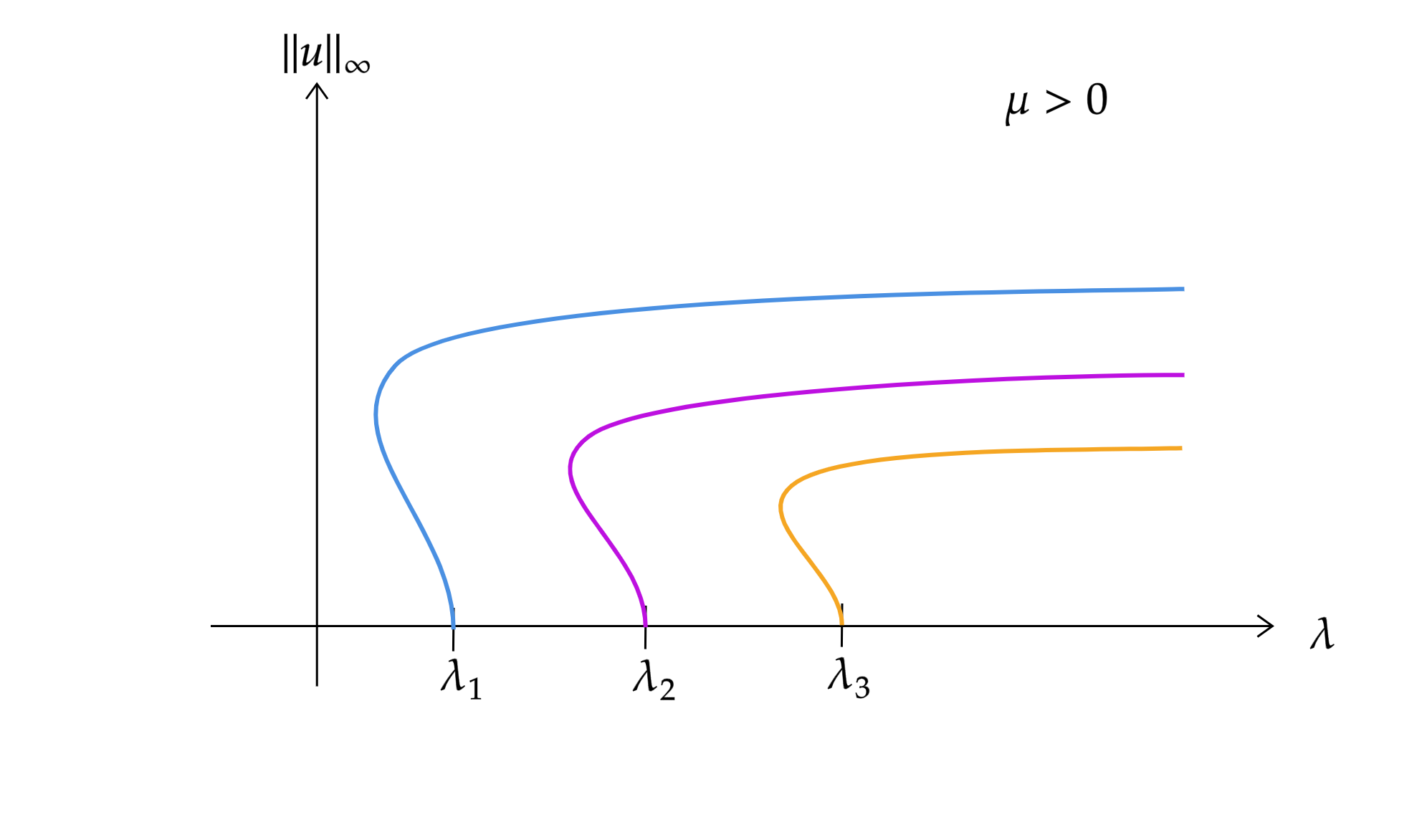}

\includegraphics[width=0.48\textwidth]{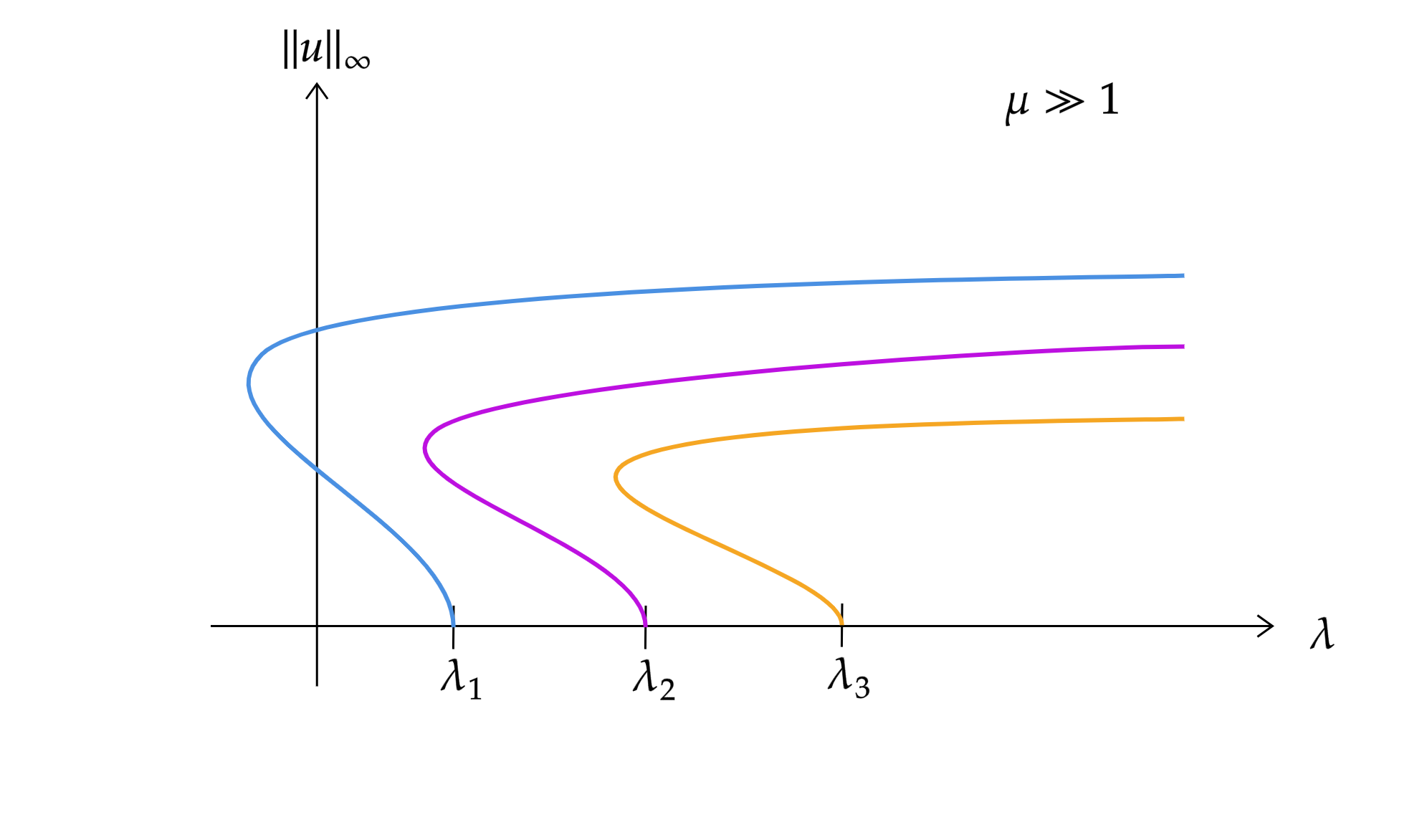}
\hspace{0.02\textwidth}
\includegraphics[width=0.48\textwidth]{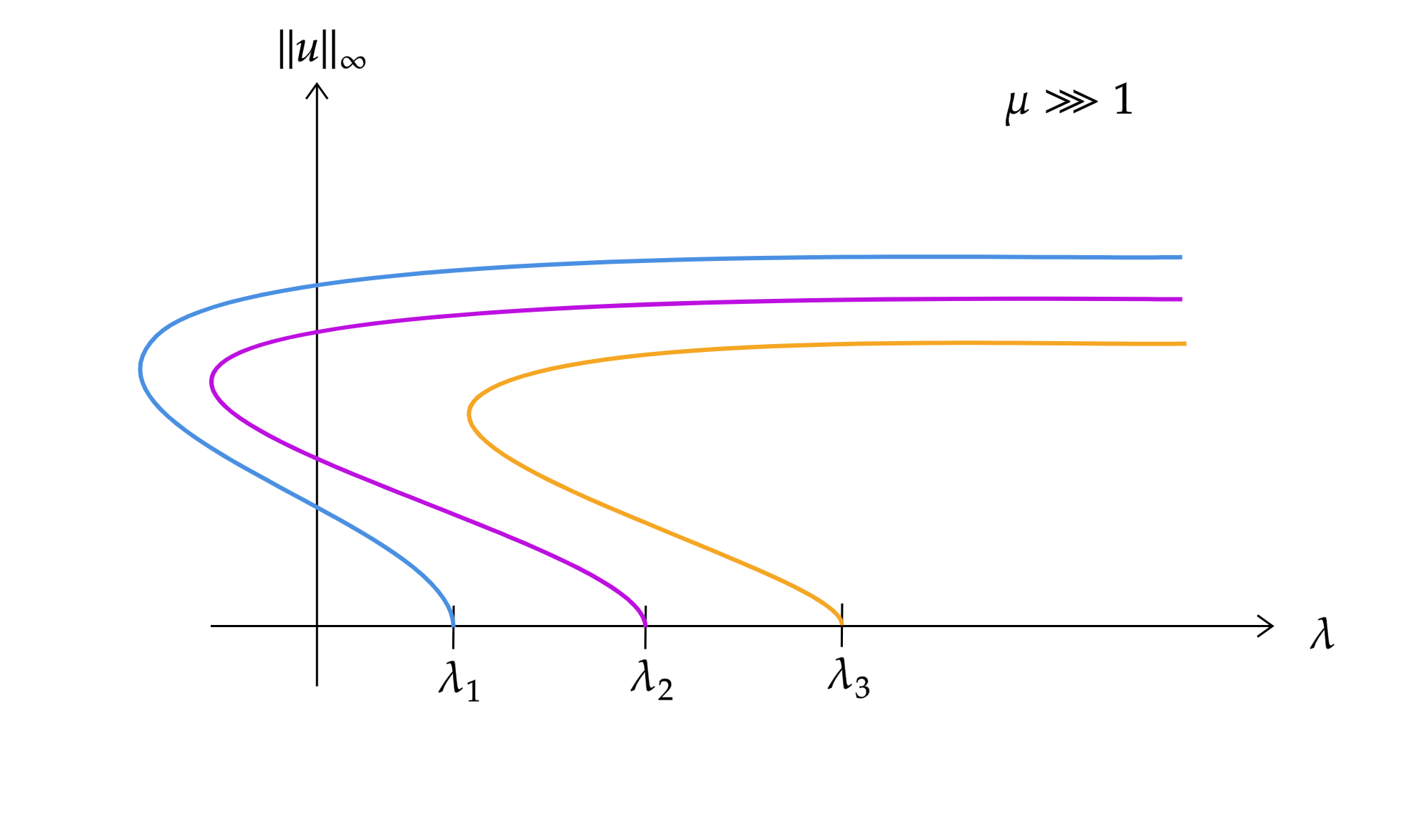}

\caption{Bifurcation diagrams with respect to $\lambda$, for different values of $\mu$.}
\label{figlambda}
\end{figure}

\appendix

\section{An auxiliary result}

In this section, we are going to present an auxiliary result dealing with a planar Hamiltonian system of the type
\begin{equation}
\label{hamiltonian}
    \begin{cases}
        x' = X(t,y) \\
        y' = -Y(t,x),
    \end{cases}
\end{equation}
where $X, Y : I \times \mathbb{R} \to \mathbb{R}$ are locally Lipschitz continuous functions. Roughly speaking, it ensures that, whenever suitable (local) sign-conditions are assumed, the number of rotations around the origin of the planar path $(x(t), y(t))$ becomes arbitrarily large as the width of the interval $I$ increases. This result was first 
proved in \cite[Lemma 3.2]{BoZa13} in the particular case 
$X(t,y) = y$ and then generalized in \cite[Proposition 2.1]{BoGa19}. We include such a version here for the sake of completeness.

In what follows, we will write (whenever possible, namely for $x(t)^2 + y(t)^2 > 0$) solutions of \eqref{hamiltonian} in (clockwise) polar coordinates as
\begin{equation}
    x(t) = \rho(t)\cos\theta(t), \qquad y(t) = -\rho(t)\sin\theta(t), \label{eq:2.2}
\end{equation}
with $\rho(t) > 0$. Notice that, of course, the angular coordinate $\theta(t)$ is defined up to integer multiples of $2\pi$; however, the expression $\theta(t_2) - \theta(t_1)$, for any $t_1, t_2 \in I$, is uniquely determined, depending on the path $(x(t), y(t))$ only.

\begin{proposition}
\label{Prop21}
Let $a_i, b_i : (-\delta, \delta) \to \mathbb{R}$, with $i = 1, 2$, be locally Lipschitz continuous functions such that
\begin{equation}
    0 < a_1(s)s \leq b_1(s)s \quad \text{and} \quad 0 < b_2(s)s \leq a_2(s)s, 
    \quad \text{for every } s \neq 0. \label{eq:2.3}
\end{equation}
Then, for every positive integer $k$, there exists $\tau_k^*>0$ and $\rho_k^* \in (0,\delta)$ such that, for every interval $I := [t_0, t_1] \subset \mathbb{R}$ with $|I| = t_1 - t_0 > \tau_k^*$ and for every locally Lipschitz continuous functions $X, Y : I \times \mathbb{R} \to \mathbb{R}$ satisfying
\begin{equation}
\label{first}
    a_1(y)y \leq X(t,y)y \leq b_1(y)y, \quad \text{for every } t \in I,\ y \in (-\delta, \delta)
\end{equation}
and
\begin{equation}
\label{second}
    b_2(x)x \leq Y(t,x)x \leq a_2(x)x, \quad \text{for every } t \in I,\ x \in (-\delta, \delta),
\end{equation}
it holds that any solution $(x(t), y(t))$ of \eqref{hamiltonian}, defined on $I$ and satisfying $x(t_0)^2 + y(t_0)^2 = (\rho_k^*)^2$, fulfills $x(t)^2 + y(t)^2 > 0$ for every $t \in I$ and
$$
    \theta(t_1) - \theta(t_0) > k\pi.
$$
\end{proposition}

\section*{Acknowledgments}
A. Boscaggin and F. Colasuonno were partially supported by GNAMPA - INdAM. R.~Ziegele gratefully acknowledges Andrea Tellini for valuable discussions on the numerical simulations.


\noindent 



\begin{thebibliography}{99}

\bibitem{Az14}
A. Azzollini,
Ground state solution for a problem with mean curvature operator in Minkowski space,
\emph{J. Funct. Anal.} 266 (2014), 2086--2095.

\bibitem{Az16}
A. Azzollini,
On a prescribed mean curvature equation in Lorentz--Minkowski space,
\emph{J. Math. Pures Appl.} (9) 106 (2016), 1122--1140.

\bibitem{BaSi82}
R. Bartnik and L. Simon,
Spacelike hypersurfaces with prescribed boundary values and mean curvature,
\emph{Comm. Math. Phys.} 87 (1982/83), 131--152.

\bibitem{BeJeTo13}
C. Bereanu, P. Jebelean, and P. J. Torres,
Positive radial solutions for Dirichlet problems with mean curvature operators
in Minkowski space,
\emph{J. Funct. Anal.} 264 (2013), 270--287.

\bibitem{BeJeMa14}
C. Bereanu, P. Jebelean, and J. Mawhin,
The Dirichlet problem with mean curvature operator in Minkowski space:
a variational approach,
\emph{Adv. Nonlinear Stud.} 14 (2014), 315--326.

\bibitem{BoCoFo19}
D.~Bonheure, F.~Colasuonno, and J.~F\"oldes,
\emph{On the Born--Infeld Equation for Electrostatic Fields with a Superposition of Point Charges},
Ann. Mat. Pura Appl. \textbf{198} (2019), no.~3, 749--772.

\bibitem{BoIa19}
D.~Bonheure and A.~Iacopetti,
\emph{On the Regularity of the Minimizer of the Electrostatic Born--Infeld Energy},
Arch. Ration. Mech. Anal. \textbf{232} (2019), 697--725.

\bibitem{Bo11}
A.~Boscaggin.
\newblock Subharmonic solutions of planar Hamiltonian systems: A rotation number approach.
\newblock {\em Adv. Nonlinear Stud.}, 11(1):77--103, 2011.

\bibitem{BoZa13}
A.~Boscaggin and F.~Zanolin.
\newblock Pairs of nodal solutions for a class of nonlinear problems with one-sided growth conditions.
\newblock {\em Adv. Nonlinear Stud.}, 13(1):13--53, 2013.

\bibitem{BoGa19}
A. Boscaggin and M. Garrione,
Pairs of nodal solutions for a Minkowski-curvature boundary value problem
in a ball,
\emph{Commun. Contemp. Math.} 21 (2019), 1850006, 18 pp.

\bibitem{BoCoNo20}
A.~Boscaggin, F.~Colasuonno, and B.~Noris.
\newblock Multiplicity of solutions for the Minkowski-curvature equation via shooting method.
\newblock {\em Bruno Pini Math. Anal. Semin.}, 11(1):1--17, 2020.

\bibitem{BoCoNo20b}
A.~Boscaggin, F.~Colasuonno, and B.~Noris.
\newblock Positive radial solutions for the Minkowski-curvature equation with Neumann boundary conditions.
\newblock {\em Discrete Contin. Dyn. Syst. Ser. S},
13(7):1921--1933, 2020.

\bibitem{ByIkMaMa24}
J.~Byeon, N.~Ikoma, A.~Malchiodi, and L.~Mari,
\emph{Existence and Regularity for Prescribed Lorentzian Mean Curvature Hypersurfaces, and the Born--Infeld Model},
Ann. PDE \textbf{10} (2024), no.~1.

\bibitem{ChYa76}
S.-Y. Cheng and S.-T. Yau,
Maximal spacelike hypersurfaces in the Lorentz-Minkowski spaces,
\emph{Ann. of Math.} 104 (1976), 407--419.

\bibitem{CoCoRi14}
I.~Coelho, C.~Corsato, and S.~Rivetti.
\newblock Positive radial solutions of the Dirichlet problem for the Minkowski-curvature equation in a ball.
\newblock {\em Topol. Methods Nonlinear Anal.}, 44(1):23--39, 2014.

\bibitem{CoObOmRi13}
C. Corsato, F. Obersnel, P. Omari, and S. Rivetti,
On the lower and upper solution method for the prescribed mean curvature
equation in Minkowski space,
\emph{Discrete Contin. Dyn. Syst.} Suppl. (2013), 159--169.

\bibitem{CoObOmRi13b}
C. Corsato, F. Obersnel, P. Omari, and S. Rivetti,
Positive solutions of the Dirichlet problem for the prescribed mean curvature
equation in Minkowski space,
\emph{J. Math. Anal. Appl.} 405 (2013), 227--239.

\bibitem{Da16}
G. Dai,
Bifurcation and positive solutions for problem with mean curvature operator
in Minkowski space,
\emph{Calc. Var. Partial Differential Equations} 55 (2016), no.~4, Paper No.~72.

\bibitem{Da17}
G. Dai,
Global bifurcation for problem with mean curvature operator on general domain,
\emph{Nonlinear Differential Equations Appl. NoDEA} 24 (2017), no.~3,
Paper No.~30.

\bibitem{DaWa17}
G. Dai and J. Wang,
Nodal solutions to problem with mean curvature operator in Minkowski space,
\emph{Differential Integral Equations} 30 (2017), no.~5--6, 463--480.

\bibitem{Ge83}
C. Gerhardt,
$H$-surfaces in Lorentzian manifolds,
\emph{Comm. Math. Phys.} 89 (1983), 523--553.

\bibitem{GiNiNi79}
B.~Gidas, W.-M.~Ni, and L.~Nirenberg.
\newblock Symmetry and related properties via the maximum principle.
\newblock {\em Comm. Math. Phys.}, 68(3):209--243, 1979.

\bibitem{MaMa25}
L.~Maniscalco and L.~Mari,
\emph{Prescribing the Mean Curvature of an Achronal Hypersurface as a Measure: The Case of 3D Spacetimes},
preprint, arXiv:2512.17670, 2025.

\bibitem{Mi21}
M. Mih\u{a}ilescu,
The spectrum of the mean curvature operator,
\emph{Proc. Roy. Soc. Edinburgh Sect. A} 151 (2021), no.~2, 451--463.

\bibitem{PaRa19}
N.~S. Papageorgiou, V.~D. R\u{a}dulescu, and D.~D. Repov\v{s},
Nonlinear Analysis---Theory and Methods,
\emph{Springer Monographs in Mathematics}. Springer, Cham, 2019.

\bibitem{Re00} 
C. Rebelo, A note on uniqueness of Cauchy problems associated to planar
\emph{Hamiltonian systems}, Portugal. Math. 57 (2000), 415--419.

\bibitem{Sz86}
A. Szulkin,
Minimax principles for lower semicontinuous functions and applications to
nonlinear boundary value problems,
\emph{Ann. Inst. H. Poincar\'e C Anal. Non Lin\'eaire} 3 (1986), no.~2,
77--109.
\end{thebibliography}
\end{document}